\documentclass[a4paper,12pt, reqno]{amsart}
\usepackage[backref]{hyperref}
\usepackage{a4wide}
\usepackage{amsthm}
 \newtheorem{theorem}{Theorem}
 
 \newtheorem{lemma}{Lemma}
 \newtheorem{corollary}{Corollary}

\begin{document}

\title{Diophantine approximation on lines in $\mathbb{C}^2$ with Gaussian prime constraints - enhanced version}
\author{Stephan Baier}
\address{Stephan Baier, Jawaharlal Nehru University, Munirka, School of Physical Sciences,
Delhi 11067, India}

\email{email\_baier@yahoo.de}


\maketitle

\begin{abstract}
We study the problem of Diophantine approximation on lines in $\mathbb{C}^2$ with numerators and denominators restricted to Gaussian primes. To 
this end, we develop analogs of well-known results on small fractional parts of $p\gamma$, $p$ running over the primes and $\gamma$ being a fixed 
irrational, for Gaussian primes. 
\end{abstract}
\maketitle


\section{Introduction}
Let $\psi : \mathbb{N} \rightarrow \mathbb{R}^+$ be a non-increasing function. 
A real number $x$ is called $\psi$-approximable if there exist infinitely many rational numbers $m/n$ with $m\in\mathbb{N}$ and 
$n\in \mathbb{Z}\setminus\{0\}$ such that
$$
\left| x- \frac{m}{n} \right| < \frac{\psi(m)}{|n|}.
$$
Khinchin's famous theorem on Diophantine approximation says
that if the series $\sum_{m} \psi(m)$ diverges, then almost every real number, in the sense of Lebesgue measure, 
is $\psi$-approximable, and if the series converges, then almost every real number is not $\psi$-approximable.
It is an interesting problem to extend problems of this kind to manifolds in place of the set of real numbers. A nice survey 
of results on Diophantine approximation on affine subspaces was given by A. Ghosh \cite{Gho}.
It is also natural to restrict the numerators and denominators in these problems to certain sets such as primes. Indeed, a version of 
Khinchin's theorem with prime numerators and denominators was proved by G. Harman \cite{Harman} and later extended to higher dimensions
by H. Jones \cite{Jones}. For plane curves of the form $y=x^{\tau}$, Harman and Jones proved the folling Khinchin-type result with prime-restrictions
in their joint work \cite{HJ}.

\begin{theorem} \label{HaJo} Let $\varepsilon>0$ and $\tau> 1$. Then for almost all positive $\alpha$ there are
infinitely many $p$, $q$, $r$, all prime, such that
\begin{equation*} 
0 < p\alpha - r \le p^{-1/6+\varepsilon} \quad \mbox{and} \quad 0< p\alpha^{\tau} - q \le p^{-1/6+\varepsilon}.
\end{equation*}
\end{theorem}

In \cite{BG1}, A. Ghosh and the author of the present paper
considered the same problem for lines in the plane, passing through the orign, establishing the following. 

\begin{theorem} \label{real} Let $\varepsilon>0$ and let $c>1$ be an irrational number. Then for almost all positive $\alpha$, with respect to the Lebesgue measure, there are infinitely many triples $(p,q,r)$ with $p$ and $r$ prime and $q$ an integer such that
\begin{equation*} 
0 < p\alpha - r \le p^{-1/5+\varepsilon} \quad \mbox{and} \quad 0< pc\alpha - q \le p^{-1/5+\varepsilon}.
\end{equation*}
\end{theorem}

Here the condition on $q$ is relaxed, i.e., we allow $q$ to be an integer, but we get a better exponent of $1/5$ in place of $1/6$.

Our approach builds on that of Harman and Jones \cite{HJ}, but we use exponential sums with prime variables instead of zero density
estimates at a particular point to get the argument work for lines. 
In \cite{BG2}, we extended this result to lines in higher dimensional spaces. 

It is very interesting to consider the same circle of problems in the setting of number fields. Indeed, a number field version of Khinchin's 
theorem was proved by D. Cantor \cite{Can}. A new proof for Khinchin's theorem in the classical case as well as the case of imaginary 
quadratic number fields was given by D. Sullivan \cite{Sul} using geodesic flows. However, it seems that number field versions of results of 
this type with prime restricitions and Diophantine approximation on general manifolds in the number field setting didn't receive much attention so far.
In this paper, we make a step into this direction by considering 
Diophantine approximation on lines in $\mathbb{C}^2$, where we restrict numerators and denominators to Gaussian primes. What we prove is
an analog of Theorem \ref{real} in the setting of the number field $\mathbb{Q}(i)$. It is likely that the method could be extended to imaginary
quadratic number fields in general, but we confine our investigation to the simplest case, which requires a large amount of
extra work and new arguments already. Our result is the following. 

\begin{theorem} \label{complex} Let $\varepsilon>0$ and let $c\in \mathbb{C}\setminus\mathbb{Q}(i)$. Then for almost all $\alpha\in \mathbb{C}$, 
with respect to the Lebesgue measure, there are infinitely many triples $(p,q,r)$ with $p$ and $r$ Gaussian primes and $q$ a Gaussian integer such that
\begin{equation} \label{simultan}
|p\alpha - r| \le |p|^{-1/12+\varepsilon} \quad \mbox{and} \quad |pc\alpha - q| \le |p|^{-1/12+\varepsilon}.
\end{equation}
\end{theorem}
 
The structure of our proof resembles that of Theorem \ref{real}, but the technical details are more
involved. In particular, we need to develop results on Diophantine approximation of numbers in $\mathbb{C}\setminus \mathbb{Q}(i)$
by fractions of Gaussian integers with Gaussian prime denominators in sectors of the complex plane, 
which is the content of section 3. As a by-product, we prove the following result on Diophantine approximation with Gaussian primes.

\begin{theorem} \label{coro} Let $c$ be a complex number such that $c\not\in \mathbb{Q}(i)$, $\varepsilon>0$ be an arbitrary constant and $-\pi\le \omega_1<
\omega_2\le \pi$. Then there exist infinitely many Gaussian primes $p$ such that
$$
\max\{||\Re(pc)||,||\Im(pc)||\}\le |p|^{-1/12+\varepsilon},
$$
where $||z||$ denotes the distance of the real number $z$ to the nearest rational integer.
\end{theorem}

We note that $\max\{||\Re(pc)||,||\Im(pc)||\}$ measures the distance of $pc$ to the nearest Gaussian integer in the sense of the supremum norm.
Thus, Theorem \ref{coro} states that this distance is infinitely often very small as $p$ runs over the Gaussian primes and 
$c\in \mathbb{C}\setminus\mathbb{Q}(i)$ is fixed. Similar results have been proved for primes in $\mathbb{Q}$ (see \cite{Mat}, for example).

A complex analog of Theorem \ref{HaJo} for $\tau\in \mathbb{N}\setminus \{1\}$, in particular the case $\tau=2$ of the complex parabola, would 
certainly be a very interesting problem to consider as well. \\   

{\bf Conventions.} 
(I) Throughout the sequel, we shall assume that $0<|c|\le 1$ in Theorem \ref{complex}. The case $|c|>1$ can be treated similarly, by minor modifications of the method.\\ 
(II) Throughout this paper, we follow the usual convention that $\varepsilon$ is a small enough positive real number.\\

{\bf Acknowledgement.} The author would like to thank Prof. Anish Ghosh for useful discussions about this topic at and after a pleasant stay
at the Tata Institute in Mumbai in August 2016. He would further like to thank the anonymous referee for useful comments on a first version
of this paper which greatly helped to make this paper self-contained.

\section{A Metrical approach}
Our basic approach for a proof of Theorem \ref{complex} is an extension of that in \cite[section 2]{BG1} (see also \cite[section 2]{BG2}) and has its origin in \cite{HJ}. We
first establish the metrical lemma below. Our proof follows closesly the arguments in \cite[Proof of Lemma 1]{HJ}.
Throughout the sequel, we denote by $\mu(\mathcal{C})$ the Lebesgue measure of a measurable set $\mathcal{C}\subseteq \mathbb{C}$ and we write
$$
D(a,b):=\left\{z\in \mathbb{C}\ :\ a< |z|\le b\right\}
$$
and 
$$
D(a,b,\gamma_1,\gamma_2):=\left\{Re^{i\theta}\ :\ a< R\le b,\ \gamma_1< \theta\le \gamma_2\right\}.
$$ 

\begin{lemma} \label{metric}
Let $\mathcal{S}$ be a subset of the positive integers. Assume that $A$ and $B$ are reals such that $0<A<B$ and let $\mathcal{M}:=D(A,B)$.
Let $F_N(\alpha)$ be a non-negative real-valued function of $N$, an element of $\mathcal{S}$,
and $\alpha$, a complex number. Let further $G_N$ and $V_N$ be real-valued functions of $N\in \mathcal{S}$ such that the following 
conditions \eqref{butzel}, \eqref{VNbound}, \eqref{FNint}, \eqref{constantK} below hold. 

\begin{equation} \label{butzel}
G_N \rightarrow \infty \quad \mbox{as } N\in \mathcal{S} \mbox{ and } N\rightarrow\infty.
\end{equation}

\begin{equation} \label{VNbound}
V_N=o\left(G_N\right) \quad \mbox{as } N\in \mathcal{S} \mbox{ and } N\rightarrow\infty. 
\end{equation}

\begin{equation} \label{FNint}
\begin{cases}
\mbox{For all } a,b,\gamma_1,\gamma_2 \mbox{ with } A\le a<b\le B \mbox{ and } -\pi\le \gamma_1<\gamma_2\le \pi \mbox{ we have}\\
\limsup\limits\limits_{\substack{N\in \mathcal{S}\\ N\rightarrow \infty}} \int\limits_{\gamma_1}^{\gamma_2} \int\limits_a^b 
\frac{F_N\left(Re^{i\theta}\right)}{G_N} \ dR\ d\theta \ge (\gamma_2-\gamma_1)\left(b^2-a^2\right).
\end{cases}
\end{equation}

\begin{equation} \label{constantK}
\begin{cases}
\mbox{There is a positive constant } K \mbox{ such that, for any measurable set } \mathcal{C}\subseteq\mathcal{M} \\
\mbox{ and any } N\in \mathcal{S},\ 
\int\limits_{\mathcal{C}} F_N\left(Re^{i\theta}\right) \ dR\ d\theta\le KG_N\mu(\mathcal{C})+V_N.
\end{cases}
\end{equation}
Then for almost all $\alpha\in \mathcal{M}$, we have 
$$
\limsup\limits_{\substack{N\in \mathcal{S}\\ N\rightarrow\infty}} \frac{F_N(\alpha)}{G_N}\ge 1. 
$$
\end{lemma}

\begin{proof} We write
$$
H_N(\alpha):=\frac{F_N(\alpha)}{G_N}
$$
and suppose that
$$
\limsup\limits_{\substack{N\in \mathcal{S}\\ N\rightarrow\infty}} H_N(\alpha)<1
$$
on a subset of $\mathcal{M}$ with positive measure. Then there must be a set $\mathcal{A}\subset \mathcal{M}$ with positive measure and a 
constant $c<1$ with
\begin{equation} \label{limsupbound}
\limsup\limits_{\substack{N\in \mathcal{S}\\ N\rightarrow\infty}} H_N(\alpha)\le c \quad \mbox{for all } \alpha\in \mathcal{A}.
\end{equation}
By the Lebesgue density theorem, for each $\varepsilon>0$ there are $a,b,\gamma_1,\gamma_2$ with $A\le a<b\le B$, 
$-\pi\le \gamma_1<\gamma_2\le \pi$ and 
$(\gamma_2-\gamma_1)(b^2-a^2)<1$
such that, if we put
$\mathcal{B}:=\mathcal{A}\cap \mathcal{Z}$ with $\mathcal{Z}:= D(a,b,\gamma_1,\gamma_2)$,
then 
$$
\mu(\mathcal{B})>(1-\varepsilon)\mu(\mathcal{Z})=(1-\varepsilon)(\gamma_2-\gamma_1)(b^2-a^2)
$$
and hence
$$
\mu(\mathcal{Z}\setminus \mathcal{B})< \varepsilon(\gamma_2-\gamma_1)(b^2-a^2)<\varepsilon.
$$
Now, using \eqref{constantK},
\begin{equation*}
\begin{split}
\int\limits_{\mathcal{Z}} H_N(\alpha) \ dR\ d\theta =& 
\int\limits_{\mathcal{B}} H_N(\alpha) \ dR\ d\theta + \int\limits_{\mathcal{Z}\setminus \mathcal{B}} H_N(\alpha) \ dR\ d\theta\\
\le & \int\limits_{\mathcal{B}} H_N(\alpha) \ dR\ d\theta +K\varepsilon+\frac{V_N}{G_N},
\end{split}
\end{equation*}
where $\arg(\alpha)=\theta$ and $|\alpha|=R$. So if 
\begin{equation*}
\varepsilon:=\frac{(1-c)(\gamma_2-\gamma_1)\left(b^2-a^2\right)}{2K},
\end{equation*}
then, in view of \eqref{VNbound} and \eqref{limsupbound}, it follows that
\begin{equation*}
\limsup\limits_{\substack{N\in \mathcal{S}\\ N\rightarrow\infty}}  \int\limits_{\mathcal{Z}} H_N(\alpha) \ dR\ d\theta \le
c\mu(\mathcal{B})+K\varepsilon=c(\gamma_2-\gamma_1)\left(b^2-a^2\right)+K\varepsilon<(\gamma_2-\gamma_1)\left(b^2-a^2\right).
\end{equation*}
This contradicts \eqref{FNint} and so completes the proof. 
\end{proof}

Now let $F_N(\alpha)$ be the number of solutions to \eqref{simultan} with $|p|\le N$ and for $0<A<B$ let
$$
G_N(A,B):=  C\cdot \frac{A}{B} \cdot \frac{N^{5/3+4\varepsilon}}{\log^2 N}, 
$$
where $C>0$ is a suitable constant only depending on $c$. 
In sections 4 and 5, we will prove the following.

\begin{theorem} \label{Theo}
There exists $C=C(c)>0$ and an infinite set $\mathcal{S}$ of natural numbers $N$ such that the following hold.\\ 

{\rm (i)} Let $0 < A < B$ be given. Then for all $a,b,\gamma_1,\gamma_2$ with $A\le a<b\le B$ and $-\pi\le \gamma_1<\gamma_2\le \pi$ we have
\begin{equation*}
\int\limits_{\gamma_1}^{\gamma_2} \int\limits_a^b F_N\left(Re^{i\theta}\right) \ dR\ d\theta \ge (\gamma_2-\gamma_1)\left(b^2-a^2\right)G_N(A,B)
\end{equation*}
if $N\in \mathcal{S}$ and $N$ large enough.\\

{\rm (ii)} Let $0<A<B$ be given. Then there exists a constant $K=K(A,B)$ such that, for every $\alpha\in \mathbb{C}$
with $A\le |\alpha|\le B$, we have 
$$
F_N(\alpha)\le KG_N(A,B)+J_N(\alpha)
$$
with
$$
\int\limits_{-\pi}^{\pi} \int\limits_A^B \left|J_N\left(Re^{i\theta}\right)\right| dR d\theta=o\left(G_N(A,B)\right) 
$$
if $N\in \mathcal{S}$ and $N\rightarrow \infty$. 
\end{theorem}

Together with Lemma \ref{metric}, this implies Theorem \ref{complex}.

\section{Diophantine approximation with Gaussian primes in sectors}
It will be of key importance to establish results related to the distribution of Gaussian primes in sectors satisfying certain Diophantine 
properties. This is the content of this section and of independent interest. 

\subsection{Results}
By $\pi(P_1,P_2,\omega_1,\omega_2)$, we denote the number of primes in $D(P_1,P_2,\omega_1,\omega_2)$. The prime number theorem for
Gaussian primes in sectors due to Kubylius \cite{Kub} implies the following.

\begin{theorem}\label{PNT} If $0\le P_1<P_2$ and $\omega_1<\omega_2\le \omega_1+2\pi$, then 
$$
\pi(P_1,P_2,\omega_1,\omega_2) = \frac{2}{\pi}\cdot \frac{(\omega_2-\omega_1)\left(P_2^2-P_1^2\right)+o\left(P_2^2\right)}{\log P_2^2}
$$
as $P_2\rightarrow \infty$. 
\end{theorem}

Further, for $\delta>0$ and $c\in \mathbb{C}$, we denote by  
$\pi_c(P_1,P_2,\omega_1,\omega_2;\delta)$ the number of  Gaussian primes $p$ contained in 
$D(P_1,P_2,\omega_1,\omega_2)$ such that 
$$
\min\limits_{q\in \mathbb{Z}[i]} |pc-q|\le \delta,
$$
and by $\pi^{\ast}_c(P_1,P_2,\omega_1,\omega_2;\delta)$ we denote the number of Gaussian primes $p$ contained in $D(P_1,P_2,\omega_1,\omega_2)$ such that
$$
\max(||\Re(pc)||,||\Im(pc)||)=\min\limits_{q\in \mathbb{Z}[i]} \max\left(|\Re(pc-q)|,|\Im(pc-q)|\right) \le \delta. 
$$
We shall establish the following theorem in the next subsections. 

\begin{theorem}\label{signi1} Let $c\in \mathbb{C}\setminus \mathbb{Q}(i)$. Then there exists an increasing sequence of natural numbers
$(M_k)_{k\in \mathbb{N}}$ such that the following holds. If $0\le P_1<P_2\le M_k$, $\omega_1<\omega_2\le \omega_1+2\pi$ and 
$M_k^{\varepsilon-1/12}<\delta_k\le 1/2$, then  
$$
\pi^{\ast}_c(P_1,P_2,\omega_1,\omega_2;\delta_k) = 
4\delta_k^2\pi^{\ast}(P_1,P_2,\omega_1,\omega_2)+o\left(\frac{\delta_k^2 M_k^2}{\log M_k}\right) 
$$
as $k\rightarrow \infty$. 
\end{theorem}

Since 
$$
\pi_c(P_1,P_2,\omega_1,\omega_2;\delta)\ge \pi^{\ast}_c(P_1,P_2,\omega_1,\omega_2;\delta/\sqrt{2}),
$$
we immediately deduce the following from Theorem \ref{PNT} and \ref{signi1}.

\begin{corollary}\label{signi} Under the conditions of Theorem \ref{signi1}, we have 
$$
\pi_c(P_1,P_2,\omega_1,\omega_2;\delta_k) \ge \frac{\delta_k^2(\omega_2-\omega_1)\left(P_2^2-P_1^2\right)+o\left(\delta_k^2 M_k^2\right)}{\log M_k}
$$
as $k\rightarrow \infty$.  
\end{corollary}

We shall employ Theorem \ref{PNT} and Corollary \ref{signi} in section 4, in which Theorem \ref{Theo}(i) is established. Partial results from 
the present section 3 are also used in section 5, in which Theorem \ref{Theo}(ii) is established. 

\subsection{Historical notes on Diophantine approximation with primes}
Let $\gamma\in \mathbb{R}$ be an irrational number. Then the continued fraction expansion of $\gamma$ yields infinitely many natural
numbers $q$ such that
$$
\left|\gamma-\frac{a}{q}\right| \le q^{-2},
$$
where $(a,q)=1$. In other words, for infinitely many $q\in \mathbb{N}$, we have
$$
||q \gamma||\le q^{-1},
$$
where $||x||$ is the distance of $x$ to the nearest integer. The problem of approximating irrational numbers by rational numbers with prime 
denominator is considerable more difficult und has a long history. The question is for which $\gamma>0$ one can prove the infinitude
of primes $p$ such that
\begin{equation} \label{approxi}
||p \gamma||\le p^{-\gamma+\varepsilon}.
\end{equation}
The first results in this direction were due to Vinogradov \cite{Vin} who showed that $\gamma=1/5$ is admissable. Vaughan \cite{Vau} improved this exponent to 
$\gamma=1/4$ using his famous identity for the von Mangoldt function. 
It should be noted that using Vaughan's method, an {\it asymptotic} result of the following form can be established. 

\begin{theorem} \label{Vaughan} Let $\gamma\in \mathbb{R}$ be irrational and $\varepsilon>0$ be an arbitrary constant. 
Then there exists an infinite 
increasing sequence of natural numbers $(N_k)_{k\in \mathbb{N}}$ such that
\begin{equation} \label{asy}
\sum\limits_{\substack{p\le N_k\\ ||p\gamma||\le \delta_k}} 1 \sim 2\delta_k \sum\limits_{p\le N_k} 1 
\quad \mbox{as } k\rightarrow \infty
\end{equation}
if 
\begin{equation} \label{range}
N_k^{-1/4+\varepsilon}\le \delta_k\le 1/2,
\end{equation}
where $p$ runs over the rational primes.
\end{theorem}

The next important step was Harman's work \cite{Har} in which he used his sieve method to show
that \eqref{approxi} holds for infinitely many primes $p$ if $\gamma=3/10$. 
Harman's method doesn't imply the asymptotic \eqref{asy} for $\delta_k=N_k^{-3/10+\varepsilon}$ since it uses a lower
bound sieve. However, Harman's sieve can be employed to recover Vaughan's result and hence \eqref{asy} for the same
$\delta_k$-range as in \eqref{range}. We further mention the work of Heath-Brown and Jia \cite{HeJ} who
used bounds for Kloosterman sums to obtain a further improvement of the exponent to $\gamma=16/49=1/3-0.0068...$. Finally,
the exponent $\gamma=1/3$ was achieved in a landmark paper by Matom\"aki \cite{Mat} who incorporated the Kuznetsov formula into the method to bound
sums of Kloosterman sums. This exponent $\gamma=1/3$ is considered to be the limit of currently available techniques.

In the following, we consider an analog problem for Gaussian primes and establish a result corresponding to 
Theorem \ref{Vaughan} in this context, thereby proving Theorem \ref{signi1}. This also implies 
the infinitude of Gaussian primes in sectors satisfying an inequality 
corresponding to \eqref{approxi}. 
To this end, we shall apply a version of Harman's sieve for 
$\mathbb{Z}[i]$. Our method will require additional counting arguments, as compared to the classical method. 
The final proof is carried out in subsection \ref{con}.

\subsection{Setup}
Throughout the following, $c$ is a fixed complex number such that $c\not \in \mathbb{Q}(i)$, and we assume that
\begin{equation} \label{assu}
1\le x_1<x_2, \quad -\pi\le \omega_1<\omega_2\le \pi \quad \mbox{and} \quad 0<\delta\le \frac{1}{2}.
\end{equation}
We compare the quantities
$$
S_c(x_1,x_2,\omega_1,\omega_2;\delta):=\sum\limits_{\substack{x_1<\mathcal{N}(p)\le x_2\\ ||pc||\le \delta\\ \omega_1<\arg p \le \omega_2}} 1 
$$
and
$$
S(x_1,x_2,\omega_1,\omega_2):=\sum\limits_{\substack{x_1<\mathcal{N}(p)\le x_2\\ \omega_1<\arg p\le \omega_2}} 1,
$$
where the sums run over Gaussian primes $p$, $\mathcal{N}(n)=|n|^2$ denotes the norm of $n\in \mathbb{Z}[i]$, and we define
$$
||z||:=\max\left\{||\Re(z)||,||\Im(z)||\right\},
$$
where $\Re(z)$ is the real part and $\Im(z)$ is the imaginary part of $z\in \mathbb{C}$. Hence, $||z||$ measures the distance of $z$ to the 
nearest Gaussian integer with respect to the supremum norm. 
We note that by Theorem \ref{PNT}, 
\begin{equation} \label{primenumber}
S(x_1,x_2,\omega_1,\omega_2) = \frac{2}{\pi} \cdot \frac{(\omega_2-\omega_1)(x_2-x_1)+o(x_2)}{\log x_2} \quad \mbox{as } x_2\rightarrow \infty.
\end{equation}

Our goal is to construct an infinite increasing sequence $\left(N_k\right)_{k\in \mathbb{N}}$ of natural numbers such that
$S_c\left(x,N_k,\omega_1,\omega_2;\delta_k\right)$ is, in a sense, well approximated by the expected quantity
$$
4\delta_k^2 S(x,N_k,\omega_1,\omega_2)
$$
if $x<N_k$ and $N_k^{-\gamma+\varepsilon}\le \delta_k\le 1/2$, where $\gamma$ is a suitable positive number.  
We shall see that $\gamma=1/24$ is admissable, which corresponds to the exponent $1/12$ in Theorem \ref{signi1}. (Recall that 
$\mathcal{N}(p)=|p|^2$.) 

\subsection{Application of Harman's sieve for $\mathbb{Z}[i]$}
In the following, let $A$ be a finite set of non-zero Gaussian integers, $G$ be a subset of the set $\mathbb{G}$ of Gaussian primes and $z$ be a positive parameter. By 
$\mathcal{S}(A,G,z)$ we denote the number
of elements of $A$ which are coprime to the product of all Gaussian primes in $G$ with norm $\le z$, i.e.
$$
\mathcal{S}(A,G,z)=\sharp\{n\in A\ :\ p\not| \; n \mbox{ for all } p\in G \mbox{ such that } \mathcal{N}(p)\le z\}.
$$
The following is a version of Harman's sieve for $\mathbb{Z}[i]$.

\begin{theorem}[Harman] \label{Harsie} Let $A,B$ be finite sets of non-zero Gaussian integers with norm $\le x$. Suppose for any sequences $(a_n)_{n\in \mathbb{Z}[i]}$ and 
$(b_n)_{n\in \mathbb{Z}[i]}$ of complex numbers satisfying $|a_n|,|b_n|\le 1$ 
the following hold:
\begin{equation} \label{t1}
\sum\limits_{\substack{\mathcal{N}(m)\le M\\ mn\in A}} a_m = \lambda \sum\limits_{\substack{\mathcal{N}(m)\le M\\ mn\in B}} a_m + O\left(Y\right),
\end{equation}
\begin{equation} \label{t2}
\sum\limits_{\substack{x^{\alpha}<\mathcal{N}(m)\le x^{\alpha+\beta}\\ mn\in A}} a_mb_n = 
\lambda \sum\limits_{\substack{x^{\alpha}<\mathcal{N}(m)\le x^{\alpha+\beta}\\ mn\in B}} a_mb_n + O\left(Y\right)
\end{equation}
for some $\lambda,Y>0$, $\alpha>0$, $0<\beta\le 1/2$ and $M>x^{\alpha}$. Then we have 
\begin{equation} \label{super}
\mathcal{S}(A,\mathbb{G},x^{\beta})=\lambda \mathcal{S}(B,\mathbb{G},x^{\beta})+O\left(Y\log^3 x\right).
\end{equation}
\end{theorem}

\begin{proof}
The proof is parallel to that of Harman's sieve for the classical case, \cite[Theorem 3.1.]{HPD} 
(Fundamental Theorem) with $R=1$ and $c_r=1$, making repeated use of the Buchstab identity in the setting of $\mathbb{Z}[i]$ (see 
\cite[Chapter 11]{HPD}). Therefore, we omit the details.
\end{proof}

In the usual terminology, the sums in \eqref{t1} are referred to as type I bilinear sums, and the sums in \eqref{t2} as type II bilinear sums.

We assume \eqref{assu} and apply Theorem \ref{Harsie} with $x=x_2$ and $\beta=1/2$ to the situation when 
\begin{eqnarray*}
A & := &\left\{n\in \mathbb{Z}[i]\ :\ x_1<\mathcal{N}(n)\le x_2,\ \omega_1<\arg n \le \omega_2,\ ||nc || \le \delta\right\},\\
B & := &\left\{n\in \mathbb{Z}[i]\ :\ x_1<\mathcal{N}(n)\le x_2,\ \omega_1<\arg n \le \omega_2\right\} \quad \mbox{and} \quad \beta=\frac{1}{2}.
\end{eqnarray*}
The parameters $\alpha$ and $M$ will later be chosen suitably. We note that 
\begin{equation} \label{Sxd}
S_c(x_1,x_2,\omega_1,\omega_2;\delta)=\sum\limits_{\substack{x_1<\mathcal{N}(p)\le x_2\\ ||nc ||\le \delta\\
\omega_1<\arg p\le \omega_2}} 1 = \mathcal{S}(A,\mathbb{G},x_2^{1/2})+O(x_2^{1/2})
\end{equation}
and
\begin{equation} \label{Sx}
S(x_1,x_2,\omega_1,\omega_2)=\sum\limits_{\substack{x_1<\mathcal{N}(p)\le x_2\\ \omega_1<\arg p\le \omega_2}} 1 = 
\mathcal{S}(B,\mathbb{G},x_2^{1/2})+O(x_2^{1/2}).
\end{equation}

\subsection{Detecting small $||nc||$}
We observe that
$$
||nc||\le \delta \Longleftrightarrow \left([\delta-\Re(nc)]-[-\delta-\Re(nc)]\right)\left([\delta-\Im(nc)]-
[-\delta-\Im(nc)]\right)=1.
$$
Hence, the type I sum in question can be written in the form
\begin{equation*}
\begin{split}
\sum\limits_{\substack{\mathcal{N}(m)\le M\\ mn\in A}} a_m = & \sum\limits_{\mathcal{N}(m)\le M} a_m \cdot
\sum\limits_{\substack{x_1/\mathcal{N}(m)<\mathcal{N}(n)\le x_2/\mathcal{N}(m)\\ \omega_1<\arg(mn)\le \omega_2}} 
\left([\delta-\Re(mnc)]-[-\delta-\Re(mnc)]\right)\times\\ & \left([\delta-\Im(mnc)]-[-\delta-\Im(mnc)]\right).
\end{split}
\end{equation*}
Further, using $[x]=x-\psi(x)-1/2$, the inner sum over $n$ can be expressed in the form
\begin{equation*} 
\begin{split}
& \sum\limits_{\substack{x_1/\mathcal{N}(m)<\mathcal{N}(n)\le x_2/\mathcal{N}(m)\\ \omega_1<\arg(mn)\le \omega_2}}  
\left([\delta-\Re(mnc)]-[-\delta-\Re(mnc)]\right)\times\\ & \left([\delta-\Im(nnc)]-[-\delta-\Im(mnc)]\right)\\
= & 4\delta^2 \sum\limits_{\substack{x_1/\mathcal{N}(m)<\mathcal{N}(n)\le x_2/\mathcal{N}(m)\\ \omega_1<\arg(mn)\le \omega_2}} 1-\\ & 2\delta 
\sum\limits_{\substack{x_1/\mathcal{N}(m)<\mathcal{N}(n)\le x_2/\mathcal{N}(m)\\ \omega_1<\arg(mn)\le \omega_2}} \left(\psi\left(\delta-\Im(mnc)\right)-\psi\left(-\delta-\Im(mnc)\right)\right) 
- \\ & 2\delta 
\sum\limits_{\substack{x_1/\mathcal{N}(m)<\mathcal{N}(n)\le x_2/\mathcal{N}(m)\\ \omega_1<\arg(mn)\le \omega_2}} \left(\psi\left(\delta-\Re(mnc)\right)-\psi\left(-\delta-\Re(mnc)\right)\right)+\\
& \sum\limits_{\substack{x_1/\mathcal{N}(m)<\mathcal{N}(n)\le x_2/\mathcal{N}(m)\\ \omega_1<\arg(mn)\le \omega_2}} \left(\psi\left(\delta-\Re(mnc)\right)-\psi\left(-\delta-\Re(mnc)\right)\right)\times\\ &   
\left(\psi\left(\delta-\Im(mnc)\right)-\psi\left(-\delta-\Im(mnc)\right)\right)\\
= &  4\delta^2 \sum\limits_{\substack{x_1/\mathcal{N}(m)<\mathcal{N}(n)\le x_2/\mathcal{N}(m)\\ \omega_1<\arg(mn)\le \omega_2}} 1 - 2\delta S_1- 2\delta S_2+ S_3, 
\end{split}
\end{equation*}
say. Next, we approximate the function $\psi(x)$ by a trigonomtrical polynomial using the following lemma due to Vaaler (see \cite{GKo}, Theorem A6). 

\begin{lemma}[Vaaler] \label{Vaaler} For $0<|t|<1$ let
$$
W(t)=\pi t(1-|t|) \cot \pi t +|t|.
$$
Fix a natural number $J$. For $x\in \mathbb{R}$ define 
$$
\psi^{\ast}(x):=-\sum\limits_{1\le |j|\le J} (2\pi i j)^{-1}W\left(\frac{j}{J+1}\right)e(jx)
$$
and
$$
\sigma(x):=\frac{1}{2J+2} \sum\limits_{|j|\le J} \left(1-\frac{|j|}{J+1}\right)e(jx).
$$
Then $\sigma(x)$ is non-negative, and we have 
$$
|\psi^{\ast}(x)-\psi(x)|\le \sigma(x)
$$
for all real numbers $x$. 
\end{lemma}

Throughout the sequel, $J$ denotes a natural number such that $J\ge \delta^{-1}$ which will be fixed in subsection \ref{finest}. From 
Lemma \ref{Vaaler}, we deduce that
\begin{equation*}
\begin{split}
S_1\ll & \frac{x_2/\mathcal{N}(m)}{J}+ \sum\limits_{1\le |j|\le J} \frac{1}{|j|} \cdot  \Big|
\sum\limits_{\substack{x_1/\mathcal{N}(m)<\mathcal{N}(n)\le x_2/\mathcal{N}(m)\\ \omega_1<\arg(mn)\le \omega_2}} \left(e\left(j(\delta-\Im(mnc))\right)-e\left(j(-\delta-\Im(mnc)\right)
\right) \Big| \\
\ll & \frac{x_2/\mathcal{N}(m)}{J}+ \sum\limits_{1\le |j|\le J} \frac{1}{|j|}\cdot \Big|
\sum\limits_{\substack{x_1/\mathcal{N}(m)<\mathcal{N}(n)\le x_2/\mathcal{N}(m)\\ \omega_1<\arg(mn)\le \omega_2}} \left(e\left(j\delta\right)-e\left(-j\delta\right)\right) \cdot 
e\left(-j\Im(mnc)\right)\Big|\\ 
\ll & \frac{x_2/\mathcal{N}(m)}{J}+ \sum\limits_{1\le |j|\le J} \min\{\delta, |j|^{-1}\} \cdot \Big|
\sum\limits_{\substack{x_1/\mathcal{N}(m)<\mathcal{N}(n)\le x_2/\mathcal{N}(m)\\ \omega_1<\arg(mn)\le \omega_2}} e\left(j\Im(mnc)\right)\Big|.
\end{split}
\end{equation*}
In a similar way, we obtain
\begin{equation*}
\begin{split}
S_2\ll \frac{x_2/\mathcal{N}(m)}{J}+ \sum\limits_{1\le |j|\le J} \min\{\delta, |j|^{-1}\} \cdot \Big|
\sum\limits_{\substack{x_1/\mathcal{N}(m)<\mathcal{N}(n)\le x_2/\mathcal{N}(m)\\ \omega_1<\arg(mn)\le \omega_2}} e\left(j\Re(mnc)\right)\Big|
\end{split}
\end{equation*}
and
\begin{equation*}
\begin{split}
S_3\ll & \frac{x_2/\mathcal{N}(m)}{J^2} + \frac{1}{J}\cdot \sum\limits_{1\le |j_1|\le J} \min\{\delta,|j_1|^{-1}\} \cdot \Big|
\sum\limits_{\substack{x_1/\mathcal{N}(m)<\mathcal{N}(n)\le x_2/\mathcal{N}(m)\\ \omega_1<\arg(mn)\le \omega_2}} e\left(j_1\Im(mnc)\right)\Big| + \\
& \frac{1}{J}\cdot \sum\limits_{1\le |j_2|\le J} \min\{\delta, |j_2|^{-1}\} \cdot \Big|
\sum\limits_{\substack{x_1/\mathcal{N}(m)<\mathcal{N}(n)\le x_2/\mathcal{N}(m)\\ \omega_1<\arg(mn)\le \omega_2}} e\left(j_2\Re(mnc)\right)\Big|+\\
& \sum\limits_{\substack{1\le |j_1|\le J\\ 1\le |j_2|\le J}} \min\{\delta, |j_1|^{-1}\} \cdot \min\{\delta, |j_2|^{-1}\} \cdot \Big|
\sum\limits_{\substack{x_1/\mathcal{N}(m)<\mathcal{N}(n)\le x_2/\mathcal{N}(m)\\ \omega_1<\arg(mn)\le \omega_2}} e\left(j_1\Im(mnc)+j_2\Re(mnc)\right)\Big|.
\end{split}
\end{equation*}
Summing over $m$ and using $|a_m|\le 1$ and $J\ge \delta^{-1}$, we get
\begin{equation} \label{typeI}
\begin{split}
\sum\limits_{\substack{\mathcal{N}(m)\le M\\ mn\in A}} a_m = & 4\delta^2 \sum\limits_{\mathcal{N}(m)\le M} a_m
\sum\limits_{\substack{x_1/\mathcal{N}(m)<\mathcal{N}(n)\le x_2/\mathcal{N}(m)\\ \omega_1<\arg(mn)\le \omega_2}} 1\\ 
& + O\left(\delta x_2^{1+\varepsilon}J^{-1}+\delta (E_1+E_2)+E_3\right)\\
= &
4\delta^2 \sum\limits_{\substack{\mathcal{N}(m)\le M\\ mn\in B}} a_m+
O\left(\delta x_2^{1+\varepsilon}J^{-1}+\delta (E_1+E_2)+E_3\right), 
\end{split}
\end{equation}
where
\begin{equation*}
E_1=\sum\limits_{1\le |j|\le J} \min\{\delta,|j|^{-1}\} \cdot \sum\limits_{\mathcal{N}(m)\le M} \Big|
\sum\limits_{\substack{x_1/\mathcal{N}(m)<\mathcal{N}(n)\le x_2/\mathcal{N}(m)\\ \omega_1<\arg(mn)\le \omega_2}} e\left(j\Im(mnc)\right)\Big|,
\end{equation*}
\begin{equation*}
E_2=\sum\limits_{1\le |j|\le J} \min\{\delta,|j|^{-1}\} \cdot \sum\limits_{\mathcal{N}(m)\le M} \Big|
\sum\limits_{\substack{x_1/\mathcal{N}(m)<\mathcal{N}(n)\le x_2/\mathcal{N}(m)\\ \omega_1<\arg(mn)\le \omega_2}} e\left(j\Re(mnc)\right)\Big|
\end{equation*}
and 
\begin{equation}
\begin{split}
E_3= & \sum\limits_{1\le |j_1|\le J} \sum\limits_{1\le |j_2|\le J} \min\{\delta,|j_1|^{-1}\} \cdot \min\{\delta,|j_2|^{-1}\} 
\times\\ & 
\sum\limits_{\mathcal{N}(m)\le M} \Big|
\sum\limits_{\substack{x_1/\mathcal{N}(m)<\mathcal{N}(n)\le x_2/\mathcal{N}(m)\\ \omega_1<\arg(mn)\le \omega_2}}e\left(j_1\Im(mnc)+j_2\Re(mnc)\right)\Big|.
\end{split}
\end{equation}

In a similar way, using $|a_m|,|b_n|\le 1$ and $J\ge \delta^{-1}$, we derive the asymptotic estimate
\begin{equation} \label{typeII}
\sum\limits_{\substack{x_2^{\alpha}<\mathcal{N}(m)\le x_2^{\alpha+\beta}\\ mn\in A}} a_mb_n = 
4\delta^2 \sum\limits_{\substack{\mathcal{N}(m)\le M\\ mn\in B}} a_mb_n+
O\left(\delta x_2^{1+\varepsilon}J^{-1}+\delta(F_1+F_2)+F_3\right), 
\end{equation}
where
\begin{equation*}
F_1= \sum\limits_{1\le |j|\le J} \min\{\delta,|j|^{-1}\} \cdot 
\Big|\sum\limits_{\substack{x_2^{\alpha}<\mathcal{N}(m)\le x_2^{\alpha+\beta}\\ mn\in A}} 
a_mb_ne\left(j\Im(mnc)\right)\Big|,
\end{equation*}
\begin{equation*}
F_2= \sum\limits_{1\le |j|\le J} \min\{\delta,|j|^{-1}\} \cdot 
\Big|\sum\limits_{\substack{x_2^{\alpha}<\mathcal{N}(m)\le x_2^{\alpha+\beta}\\ mn\in A}} 
a_mb_ne\left(j\Re(mnc)\right)\Big|
\end{equation*}
and 
\begin{equation}
\begin{split}
F_3= & \sum\limits_{1\le |j_1|\le J} \sum\limits_{1\le |j_2|\le J} \min\{\delta,|j_1|^{-1}\} \cdot \min\{\delta,|j_2|^{-1}\}\times\\
& 
\Big|\sum\limits_{\substack{x_2^{\alpha}<\mathcal{N}(m)\le x_2^{\alpha+\beta}\\ mn\in A}} a_mb_ne\left(j_1\Im(mnc)+j_2\Re(mnc)\right)\Big|.
\end{split}
\end{equation}

\subsection{Transformations of the sums $E_i$ and $F_i$} \label{trans}

We note that 
\begin{equation} \label{E1E2}
E_1=E_2.
\end{equation} 
We further have, by breaking the $|j|$-range into $O(\log 2J)$ dyadic intervals, 
\begin{equation} \label{E1H}
E_1\ll (\log 2J) \cdot \sup\limits_{1\le H\le J} \min\{\delta,H^{-1}\} \cdot E_1(H),
\end{equation}
where
\begin{equation}
E_1(H)=\sum\limits_{1\le |j|\le H} \sum\limits_{\mathcal{N}(m)\le M} \Big|
\sum\limits_{\substack{x_1/\mathcal{N}(m)<\mathcal{N}(n)\le x_2/\mathcal{N}(m)\\ \omega_1<\arg(mn)\le \omega_2}} e\left(j\Im(mnc)\right)\Big|.
\end{equation}
Similarly,
\begin{equation} \label{E3rel}
E_3\ll (\log 2J)^2\cdot \sup\limits_{\substack{1\le |H_1|\le J\\ 1\le |H_2|\le J}} \min\{\delta,H_1^{-1}\} \cdot
\min\{\delta,H_2^{-1}\} \cdot E_3(H_1,H_2),
\end{equation}
where 
\begin{equation}
E_3(H_1,H_2)=\mathop{\sum\limits_{|j_1|\le H_1} 
\sum\limits_{|j_2|\le H_2}}_{(j_1,j_2)\not= (0,0)} \sum\limits_{\mathcal{N}(m)\le M} \Big|
\sum\limits_{\substack{x_1/\mathcal{N}(m)<\mathcal{N}(n)\le x_2/\mathcal{N}(m)\\ \omega_1<\arg(mn)\le \omega_2}}e\left(j_1\Im(mnc)+j_2\Re(mnc)\right)\Big|.
\end{equation}
We note that
\begin{equation} \label{Edef}
E_3(H_1,H_2)=\sum\limits_{\substack{j\not=0\\ |\Re(j)|\le H_1\\ |\Im(j)|\le H_2}} \sum\limits_{\mathcal{N}(m)\le M}
\Big| \sum\limits_{\substack{x_1/\mathcal{N}(m)<\mathcal{N}(n)\le x_2/\mathcal{N}(m)\\ \omega_1<\arg(mn)\le \omega_2}} e\left(\Im(jmnc)\right)\Big|
\end{equation}
and hence, 
\begin{equation} \label{E1E3}
E_1(H)= E_3(H,1/2). 
\end{equation}
Thus, it suffices to estimate $E_3(H_1,H_2)$ for $H_1\ge 1$ and $H_2\ge 1/2$ to bound $E_1$, $E_2$ and $E_3$.

Similarly,
\begin{equation} \label{F1F2}
F_1=F_2
\end{equation}
and 
\begin{equation} \label{F1H}
F_1\ll (\log 2J) \cdot \sup\limits_{1\le H\le J} \min\{\delta,H^{-1}\} \cdot F_1(H),
\end{equation}
where
\begin{equation}
F_1(H)=\sum\limits_{1\le |j|\le H} \Big|\sum\limits_{\substack{x_2^{\alpha}<\mathcal{N}(m)\le x_2^{\alpha+\beta}\\ mn\in A}} 
a_mb_ne\left(j\Im(mnc)\right)\Big|,
\end{equation}
and 
\begin{equation} \label{F3rel}
F_3= (\log 2J)^2\cdot \sup\limits_{\substack{1\le |H_1|\le J\\ 1\le |H_2|\le J}} \min\{\delta,H_1^{-1}\} \cdot
\min\{\delta,H_2^{-1}\} \cdot F_3(H_1,H_2),
\end{equation}
where 
\begin{equation} \label{F3}
\begin{split}
F_3(H_1,H_2)= & \mathop{\sum\limits_{|j_1|\le H_1} \sum\limits_{|j_2|\le H_2}}_{(j_1,j_2)
\not=(0,0)}
\Big|\sum\limits_{\substack{x_2^{\alpha}<\mathcal{N}(m)\le x_2^{\alpha+\beta}\\ mn\in A}} 
a_mb_ne\left(j_1\Im(mnc)+j_2\Re(mnc)\right)\Big|\\
= & \sum\limits_{\substack{j\not=0\\ |\Re(j)|\le H_1\\ |\Im(j)|\le H_2}} \Big|\sum\limits_{\substack{x_2^{\alpha}<\mathcal{N}(m)\le x_2^{\alpha+\beta}\\ mn\in A}} 
a_mb_ne\left(\Im(jmnc)\right)\Big|
\end{split}
\end{equation}
and hence,
\begin{equation} \label{F1F3}
F_1(H)=F_3(H,1/2).
\end{equation}
Thus, it suffices to estimate $F_3(H_1,H_2)$ for $H_1\ge 1$ and $H_2\ge 1/2$ to bound $F_1$, $F_2$ and $F_3$.

So we have reduced the problem to bounding the type I sums $E_3(H_1,H_2)$ and the type II sums $F_3(H_1,H_2)$.

\subsection{Treatment of type II sums} \label{treat}
To treat the type II sums, we first reduce them to type I sums. We begin by splitting $F_3(H_1,H_2)$ into subsums of the form
\begin{equation} \label{subsums}
F_3(H_1,H_2,K,K'):=\sum\limits_{\substack{j\not=0\\ |\Re(j)|\le H_1\\ |\Im(j)|\le H_2}} \Big|\sum\limits_{\substack{K<\mathcal{N}(m)\le K'\\ mn\in A}} 
a_mb_ne\left(\Im(jmnc)\right)\Big|,
\end{equation}
where $K<K'\le 2K$. Next, we apply the Cauchy-Schwarz inequality, getting
$$
F_3(H_1,H_2,K,K')^2\ll H_1H_2K\cdot \sum\limits_{\substack{j\not=0\\ |\Re(j)|\le H_1\\ |\Im(j)|\le H_2}} 
\sum\limits_{\substack{K<\mathcal{N}(m)\le K'}} \Big| \sum\limits_{\substack{x_1/\mathcal{N}(m)<\mathcal{N}(n)\le x_2/\mathcal{N}(m)
\\ \omega_1<\arg(mn)\le \omega_2}} 
b_ne\left(\Im(jmnc)\right)\Big|^2, 
$$
where we use the bound $|a_m|\le 1$. 
Expanding the square and re-arranging summation, we get
\begin{equation} \label{F}
\begin{split}
& F_3(H_1,H_2,K,K')^2\ll H_1H_2K\cdot  \sum\limits_{\substack{j\not=0\\ |\Re(j)|\le H_1\\ |\Im(j)|\le H_2}}\sum\limits_{x_1/K'<\mathcal{N}(n_1),\mathcal{N}(n_2) \le x_2/K} b_{n_1}\overline{b_{n_2}}\times\\ 
& \sum\limits_{\substack{\max\{K,x_1/\mathcal{N}(n_1),x_1/\mathcal{N}(n_2)\}<\mathcal{N}(m)\le \min\{K',x_2/\mathcal{N}(n_1),x_2/\mathcal{N}(n_2)\}\\ \omega_1<\arg(mn_{1,2})\le \omega_2}}
e\left(\Im(jm(n_1-n_2)c)\right)\\
\ll & H_1^2H_2^2Kx+ H_1H_2K\cdot \sum\limits_{\substack{j\not=0\\ |\Re(j)|\le H_1\\ |\Im(j)|\le H_2}} 
\sum\limits_{\substack{x_1/K'<\mathcal{N}(n_1),\mathcal{N}(n_2) \le x_2/K\\ n_1\not=n_2}} 
\\ 
& \Big| \sum\limits_{\substack{\max\{K,x_1/\mathcal{N}(n_1),x_1/\mathcal{N}(n_2)\}<\mathcal{N}(m)\le \min\{K',x_2/\mathcal{N}(n_1),x_2/\mathcal{N}(n_2)\}\\ \omega_1<\arg(mn_{1,2})\le \omega_2}}
e\left(\Im(jm(n_1-n_2)c)\right)\Big|\\
\ll & H_1^2H_2^2Kx+ H_1H_2K\cdot \sum\limits_{0<\mathcal{N}(n)\le 4(H_1^2+H_2^2)x_2/K} \sum\limits_{j|n} 
\sum\limits_{\substack{x_1/K'<\mathcal{N}(n_1),\mathcal{N}(n_2) \le x_2/K\\ n/j=n_1-n_2}} \\
& \Big| \sum\limits_{\substack{\max\{K,x_1/\mathcal{N}(n_1),x_1/\mathcal{N}(n_2)\}<\mathcal{N}(m)\le \min\{K',x_2/\mathcal{N}(n_1),x_2/\mathcal{N}(n_2)\}\\ \omega_1<\arg(mn_{1,2})\le \omega_2}}
e\left(\Im(mnc)\right)\Big|.\\
\end{split}
\end{equation}
Here the second line arrives by isolating the diagonal contribution of $n_1=n_2$ and using the bound $|b_n|\le 1$,
and the third line arrives by writing $n=j(n_1-n_2)$.

We note that the summation condition $\omega_1<\arg(mn_{1,2})\le \omega_2$ is equivalent to $f_1<\arg m\le f_2$ for suitable $f_1$ and $f_2$ 
depending on $n_1,n_2,\omega_1,\omega_2$. 

\subsection{Estimating sums of linear exponential sums}
Our next task is to bound linear exponential sums of the form 
\begin{equation} \label{1}
\begin{split}
& \sum\limits_{\substack{\tilde{y}<\mathcal{N}(m)\le y\\ f_1<\arg m\le f_2}} e\left(\Im(m\kappa)\right) =  
\sum\limits_{\substack{(m_1,m_2)\in \mathbb{Z}^2\\ 
\tilde{y}<m_1^2+m_2^2\le y\\ f_1<\arg(m_1+im_2)\le f_2}} 
e\left(\Re(\kappa)m_2+\Im(\kappa)m_1\right)\\ = & \sum\limits_{\substack{(m_1,m_2)\in \mathbb{Z}^2\\ m_1^2+m_2^2\le y\\ f_1<\arg(m_1+im_2)\le f_2}} 
e\left(\Re(\kappa)m_2+\Im(\kappa)m_1\right) - \sum\limits_{\substack{(m_1,m_2)\in \mathbb{Z}^2\\ m_1^2+m_2^2\le \tilde{y}\\ f_1<\arg(m_1+im_2)\le f_2}} 
e\left(\Re(\kappa)m_2+\Im(\kappa)m_1\right),
\end{split}
\end{equation}
where $\kappa$ is a complex number and $0\le \tilde{y}<y$. Here we use the following simple slicing argument. We have
\begin{equation} \label{m_1m_2}
\begin{split} 
& \sum\limits_{\substack{(m_1,m_2)\in \mathbb{Z}^2\\ 
m_1^2+m_2^2\le y\\ f_1<\arg(m_1+im_2)\le f_2}} 
e\left(\Re(\kappa)m_2+\Im(\kappa)m_1\right)\\ = & \sum\limits_{-\sqrt{y}\le m_1\le \sqrt{y}} 
\sum\limits_{\substack{-\sqrt{y-m_1^2}\le  m_2 \le \sqrt{y-m_1^2}\\ f_1<\arg(m_1+im_2)\le f_2}}
e\left(\Re(\kappa)m_2+\Im(\kappa)m_1\right) \\
\ll & \sum\limits_{-\sqrt{y}\le m_1\le \sqrt{y}}  \Big|\sum\limits_{\substack{-\sqrt{y-m_1^2}\le  m_2 \le \sqrt{y-m_1^2}\\ f_1<\arg(m_1+im_2)\le f_2}} 
e\left(\Re(\kappa)m_2\right) \Big|\\
\ll & y^{1/2} \cdot \min\left\{|\Re(\kappa)||^{-1},\sqrt{y}\right\},
\end{split}
\end{equation}
where we observe that the summation condition
$$
-\sqrt{y-m_1^2}\le  m_2 \le \sqrt{y-m_1^2},\quad f_1<\arg(m_1+im_2)\le f_2
$$
on $m_2$ is equivalent to $m_2\in I$ for some interval $I\subseteq [-\sqrt{y},\sqrt{y}]$ depending on $m_1,f_1,f_2$ and 
use the classical bound
$$
\sum\limits_{a<m\le b} e(mz) \ll \min\left\{b-a+1,||z||^{-1}\right\}
$$
for linear exponential sums. Similarly, by interchanging the rules of $m_1$ and $m_2$, we get
\begin{equation*}
\begin{split} 
 \sum\limits_{\substack{(m_1,m_2)\in \mathbb{Z}^2\\ 
m_1^2+m_2^2\le y\\ f_1<\arg(m_1+im_2)\le f_2}} 
e\left(\Re(\kappa)m_2+\Im(\kappa)m_1\right) \ll y^{1/2} \cdot \min\left\{||\Im(\kappa)||^{-1},\sqrt{y}\right\}.
\end{split}
\end{equation*}
Taking the geometric mean of these two estimates gives 
\begin{equation*}
\sum\limits_{\substack{(m_1,m_2)\in \mathbb{Z}^2\\ m_1^2+m_2^2\le y\\ f_1<\arg(m_1+im_2)\le f_2}} 
e\left(\Re(\kappa)m_2+\Im(\kappa)m_1\right) \ll y^{1/2} \cdot 
\min\left\{||\Im(\kappa)||^{-1},\sqrt{y}\right\}^{1/2} \cdot \min\left\{||\Re(\kappa)||^{-1},\sqrt{y}\right\}^{1/2}.
\end{equation*}
Using \eqref{1}, we deduce that
\begin{equation} \label{lin}
\sum\limits_{\substack{\tilde{y}<\mathcal{N}(m)\le y\\ f_1<\arg m\le f_2}} e\left(\Im(m\kappa)\right) \ll y^{1/2} \cdot 
\min\left\{||\Im(\kappa)||^{-1},\sqrt{y}\right\}^{1/2} \cdot \min\left\{||\Re(\kappa)||^{-1},\sqrt{y}\right\}^{1/2}.
\end{equation}

To bound the sums appearing in subsections \ref{trans} and \ref{treat}, we need to bound sums of linear sums of roughly the shape
$$
\sum\limits_{\mathcal{N}(n)\sim Z} \left|\sum\limits_{\substack{\mathcal{N}(m)\sim Y\\ f_1(n)<\arg m\le f_2(n)}} e\left(\Im(mnc)\right)\right|.
$$
Considering \eqref{lin}, we  are left with bounding expressions of the form
\begin{equation} \label{Gc}
G_{c}(y,z):=\sum\limits_{0<\mathcal{N}(n)\le z} \min\left\{||\Im(n c)||^{-1},\sqrt{y}\right\}^{1/2} \cdot 
\min\left\{||\Re(n c)||^{-1},\sqrt{y}\right\}^{1/2},
\end{equation}
where $y,z\ge 1$. To this end, we break the above into partial sums 
\begin{equation}
\begin{split}
& G_{c}(y,z,\Delta_1,\Delta_1',\Delta_2,\Delta_2')\\ := & \sum\limits_{\substack{0<\mathcal{N}(n)\le z\\ \Delta_1< ||\Im(n c)|| \le \Delta_1'\\ 
\Delta_2< ||\Re(n c)||\le \Delta_2'}} 
\min\left\{||\Im(n c)||^{-1},\sqrt{y}\right\}^{1/2} \cdot \min\left\{||\Re(n c)||^{-1},\sqrt{y}\right\}^{1/2} 
\end{split}
\end{equation}
with $0\le \Delta_1< \Delta_1'\le 1/2$ and $0\le \Delta_2< \Delta_2'\le 1/2$ and bound them by 
\begin{equation} \label{G}
G_{c}(y,z,\Delta_1,\Delta_1',\Delta_2,\Delta_2')\ll  
\min\left\{\Delta_1^{-1},\sqrt{y}\right\}^{1/2} \cdot \min\left\{\Delta_2^{-1},\sqrt{y}\right\}^{1/2} \cdot  
\Sigma_{c}(z,\Delta_1',\Delta_2'),
\end{equation}
where
\begin{equation}\label{Sigma}
\Sigma_{c}(z,\Delta_1',\Delta_2') = 
\sum\limits_{\substack{0<\mathcal{N}(n)\le z\\ ||\Im(n c)|| \le \Delta_1'\\ ||\Re(n c)||
\le \Delta_2'}} 1.
\end{equation}

In the next subsection, we shall prove that for infinitely many Gaussian integers $q$, a bound of the form 
\begin{equation} \label{plug}
\Sigma_{c}(z,\Delta_1',\Delta_2')\ll  
\left(1+\frac{z}{|q|^2}\right) \cdot \left(1+\Delta_1'|q|\right)\left(1+\Delta_2'|q|\right)
\end{equation}
holds.  We shall also see that for these $q$, we have
\begin{equation} \label{also}
\Sigma_{c}(z,\Delta_1',\Delta_2')=0 \quad \mbox{if }
\max\{\Delta_1',\Delta_2'\}< 1/(\sqrt{8}|q|) \mbox{ and } z\le |q|^2/8. 
\end{equation}
Plugging \eqref{plug} into \eqref{G} gives
\begin{equation} \label{Gtheta}
\begin{split}
G_{c}(y,z,\Delta_1,\Delta_1',\Delta_2,\Delta_2')\ll &  
\left(1+\frac{z}{|q|^2}\right)\cdot \min\left\{\Delta_1^{-1},\sqrt{y}\right\}^{1/2} \cdot \min\left\{\Delta_2^{-1},\sqrt{y}\right\}^{1/2} \times\\ &  
\left(1+\Delta_1'|q|\right)\cdot \left(1+\Delta_2'|q|\right).
\end{split}
\end{equation}

Next, we write
\begin{equation*}
\begin{split}
G_{c}(y,z)=& G_{c}(y,z,0,2^{-L-1},0,2^{-L-1})+\sum\limits_{i=1}^L \sum\limits_{j=1}^L G_{c}(y,z,2^{-i-1},2^{-i},2^{-j-1},2^{-j})+\\ & + \sum\limits_{j=1}^L 
G_{c}(y,z,0,2^{-L-1},2^{-j-1},2^{-j})+\sum\limits_{i=1}^L 
G_{c}(y,z,2^{-i-1},2^{-i},0,2^{-L-1}),
\end{split}
\end{equation*}
where $L$ satisfies $1/(2\sqrt{y})\le 2^{-L-1}< 1/\sqrt{y}$. Using \eqref{Gtheta}, we deduce that
\begin{equation} \label{firstcase}
\begin{split}
G_{c}(y,z)\ll& \left(1+\frac{z}{|q|^2}\right)\left(y^{1/2}+|q|^2\right) \cdot (\log 2y)^2.
\end{split}
\end{equation}
If $z\le |q|^2/8$, then using \eqref{also}, we have 
\begin{equation*}
\begin{split}
G_{c}(y,z)=& \sum\limits_{i=1}^L \sum\limits_{j=1}^L G_{c}(y,z,2^{-i-1},2^{-i},2^{-j-1},2^{-j})+\\ & + \sum\limits_{j=1}^L 
G_{c}(y,z,0,2^{-L-1},2^{-j-1},2^{-j})+\sum\limits_{i=1}^L 
G_{c}(y,z,2^{-i-1},2^{-i},0,2^{-L-1}),
\end{split}
\end{equation*}
where $1/(2\sqrt{8}|q|)\le 2^{-L-1}< 1/(\sqrt{8}|q|)$. In this case, using \eqref{Gtheta}, we deduce that
\begin{equation} \label{secondcase}
\begin{split}
G_{c}(y,z)\ll \left(|q|y^{1/4}+|q|^{2}\right)\cdot \log^2(2|q|).
\end{split}
\end{equation}

\subsection{Counting}
In this subsection, we prove \eqref{plug} and \eqref{also}. To bound the quantity $G_{c}(y,z)$, 
we need information about the spacing of the points $nc$ modulo 1, where 
$n\in \mathbb{Z}[i]$.   
We begin by using the Hurwitz continued fraction development of $c$ in $\mathbb{Z}[i]$ (see \cite{Hur}) to approximate $c$ in the form 
$$
c=\frac{a}{q}+\gamma,
$$
where $a,q\in \mathbb{Z}[i]$, $(a,q)=1$ and 
$$
|\gamma| \le |q|^{-2}.
$$
As in the classical case, this continued fraction development yields a sequence of infinitely many $q\in \mathbb{Z}[i]$ satisfying the above. 
Now it follows that 
\begin{equation} \label{dist}
\begin{split}
& \left|\left|n_1c-n_2c\right|\right|=\left|\left| \frac{(n_1-n_2)a}{q} + (n_1-n_2)\gamma\right|\right| \ge
\left|\left| \frac{(n_1-n_2)a}{q} \right|\right| - |n_1-n_2|\cdot |\gamma|\\
\ge & \frac{1}{\sqrt{2}|q|}- \frac{|n_1-n_2|}{|q|^2} 
\end{split}
\end{equation}
if $n_1,n_2\in \mathbb{Z}[i]$ such that $n_1\not\equiv n_2\bmod{q}$. We cover the set 
$$
\mathcal{Z}:=\{n\in \mathbb{Z}[i]\ :\ 0<\mathcal{N}(n)\le z\}
$$
by $O\left(1+z/|q|^2\right)$ disjoint rectangles 
$$
\mathcal{R}=\{s\in \mathbb{C}\ :\ a_1< \Re(s)\le b_1,\ a_2< \Im(s)\le b_2\},
$$
where $|b_i-a_i|\le |q|/4$, so that 
$$
\mathcal{Z}\subset \bigcup\limits_{\mathcal{R}} \mathcal{R}.
$$ 
Note that if $n_1,n_2\in \mathbb{Z}[i]\cap \mathcal{R}$, then $|n_1-n_2|\le |q|/(2\sqrt{2})$ and
hence, by \eqref{dist}, if $n_1,n_2\in \mathbb{Z}[i]\cap \mathcal{R}$ and $n_1\not=n_2$, then 
\begin{equation} \label{dist2}
\begin{split}
\left|\left|n_1c-n_2c\right|\right|\ge \frac{1}{2\sqrt{2}|q|}. 
\end{split}
\end{equation}

Now, 
\begin{equation*}
\Sigma_{c}(z,\Delta_1',\Delta_2') \le \sum\limits_{\mathcal{R}} \Sigma_{c}(\mathcal{R},\Delta_1',\Delta_2'),
\end{equation*}
where 
\begin{equation} \label{of}
\begin{split}
& \Sigma_{c}(\mathcal{R},\Delta_1',\Delta_2') := 
\sum\limits_{\substack{n\in \mathbb{Z}[i]\cap \mathcal{R}\\ ||\Im(n c)||\le \Delta_1'\\ ||\Re(n c)|| \le \Delta_2'}} 1 \\
= & \sum\limits_{\substack{n\in \mathbb{Z}[i]\cap\mathcal{R}\\ \{\Im(n c)\}\le \Delta_1'\\ \{\Re(n c)\} \le \Delta_2'}} 1 +
\sum\limits_{\substack{n\in \mathbb{Z}[i]\cap\mathcal{R}\\ \{\Im(n c)\}\ge 1-\Delta_1'\\ \{\Re(n c)\} \le \Delta_2'}} 1 +
\sum\limits_{\substack{n\in \mathbb{Z}[i]\cap\mathcal{R}\\ \{\Im(n c)\}\le \Delta_1'\\ \{\Re(n c)\} \ge 1-\Delta_2'}} 1 +
\sum\limits_{\substack{n\in \mathbb{Z}[i]\cap\mathcal{R}\\ \{\Im(n c)\}\ge 1-\Delta_1'\\ \{\Re(n c)\} \ge 1-\Delta_2'}} 1.
\end{split}
\end{equation}
If $\{\Im(n_i c)\}\le \Delta_1'\le 1/2$ and $\{\Re(n_i c)\} \le \Delta_2' \le 1/2$ for $i=1,2$, then
$$
\left|(\{\Re(n_1c)\}+i\{\Im(n_1c)\})-(\{\Re(n_2c)\}+i\{\Im(n_2c)\})\right|\ge
\left|\left|n_1c-n_2c\right|\right|,
$$
and hence, by \eqref{dist2}, if $n_1,n_2\in \mathbb{Z}[i]\cap \mathcal{R}$ and $n_1\not=n_2$, then
$$
\left|(\{\Re(n_1c)\}+i\{\Im(n_1c)\})-(\{\Re(n_2c)\}+i\{\Im(n_2c)\})\right|\ge \frac{1}{2\sqrt{2}|q|}.
$$
It follows that 
$$
\sum\limits_{\substack{n\in \mathbb{Z}[i]\cap\mathcal{R}\\ \{\Im(n c)\}\le \Delta_1'\\ \{\Re(n c)\} \le \Delta_2'}} 1\ll 
V_{1/(2\sqrt{2}|q|)}\left(\Delta_1',\Delta_2'\right),
$$
where $V_D\left(\Delta_1'\Delta_2'\right)$ is the maximal number of points of distance $\ge D$ that can be put 
into a rectangle with dimensions $\Delta_1'$ and $\Delta_2'$. The remaining three sums in the last line of \eqref{of} can be
estimated similarly. It follows that
$$
\Sigma_{c}(\mathcal{R},\Delta_1',\Delta_2')\ll V_{1/(2\sqrt{2}|q|)}\left(\Delta_1',\Delta_2'\right).
$$
Clearly,
$$
V_D\left(\Delta_1',\Delta_2'\right)\ll \left(1+\frac{\Delta_1'}{D}\right)\left(1+\frac{\Delta_2'}{D}\right). 
$$ 
Putting everything together, we obtain \eqref{plug}. Further, \eqref{also} holds
because $0<\mathcal{N}(n)\le |q|^2/8$ implies
\begin{equation} 
\left|\left|nc\right|\right|=\left|\left| \frac{na}{q}+ n\gamma\right|\right| \ge
\left|\left| \frac{na}{q}\right|\right| - |n|\cdot |\gamma|\ge \frac{1}{\sqrt{2}|q|}- \frac{\sqrt{|q|^2/8}}{|q|^2}=\frac{1}{\sqrt{8}|q|}. 
\end{equation} 

\subsection{Final estimations of the sums $E_i$ and $F_i$} \label{finest}
In this section, we set $x:=x_2$ for simplicity. We recall the conditions $H_1\ge 1$ and $H_2\ge 1/2$. Combining \eqref{F}, \eqref{lin}, \eqref{Gc} and \eqref{firstcase}, we get
\begin{equation*} 
\begin{split}
& F_3(H_1,H_2,K,K')^2 \\
\ll & (H_1H_2x)^{\varepsilon}\cdot \left(H_1^2H_2^2xK+ H_1H_2xK^{1/2}\cdot 
\left(1+\frac{\left(H_1^2+H_2^2\right)x/K}{|q|^2}\right)\left(K^{1/2}+|q|^2\right)\right),
\end{split}
\end{equation*}
where we use the facts that the number $\tau(n)$ of divisors $j$ of $n$ is $O\left(\mathcal{N}(n)^{\varepsilon}\right)$ and that the number of solutions 
$(n_1,n_2)$ with $x_1/(2K')<\mathcal{N}(n_1),\mathcal{N}(n_2)\le x_2/K$ of the equation $n/j=n_1-n_2$ is $O(x/K)$. 
Multiplying out and taking square root yields
\begin{equation} \label{F3part} 
\begin{split}
F_3(H_1,H_2,K,K')
\ll & (H_1H_2x)^{\varepsilon}\cdot \Big(H_1H_2(xK)^{1/2}+ (H_1H_2)^{1/2} \times\\
& \left((H_1+H_2)x|q|^{-1}+(H_1+H_2)x K^{-1/4}+|q| x^{1/2}K^{1/4}\right)\Big).
\end{split}
\end{equation}
Recall the definition of $F_3(H_1,H_2)$ in \eqref{F3}. From \eqref{F3part}, we conclude that
\begin{equation} \label{F3est} 
\begin{split}
F_3(H_1,H_2)
\ll & (H_1H_2x)^{\varepsilon}\cdot \Big(H_1H_2x^{(1+\alpha+\beta)/2}+ (H_1H_2)^{1/2} \times\\
& \left((H_1+H_2)x|q|^{-1}+(H_1+H_2)x^{1-\alpha/4}+|q|x^{1/2+(\alpha+\beta)/4}\right)\Big)
\end{split}
\end{equation}
by splitting the summation range of $\mathcal{N}(m)$ into $O(\log 2x)$ dyadic intervals $(K,K']$. 

We also split $E_3(H_1,H_2)$, defined in \eqref{Edef}, into $O(\log 2M)$ parts
\begin{equation*} 
E_3(H_1,H_2,K,K') := \sum\limits_{\substack{j\not=0\\ |\Re(j)|\le H_1\\ |\Im(j)|\le H_2}} \sum\limits_{K<\mathcal{N}(m)\le K'}
\Big| \sum\limits_{\substack{x_1/\mathcal{N}(m)<\mathcal{N}(n)\le x_2/\mathcal{N}(m)\\ \omega_1<\arg(mn)\le \omega_2}} e\left(\Im(jmnc)\right)\Big|
\end{equation*}
with $1/2\le K<K'\le 2K$, which, using \eqref{lin}, \eqref{Gc} and \eqref{firstcase}, we estimate by
\begin{equation}
\begin{split}
& E_3(H_1,H_2,K,K')\\ \ll & \left(x/K\right)^{1/2}\cdot
\sum\limits_{\substack{j\not=0\\ |\Re(j)|\le H_1\\ |\Im(j)|\le H_2}} \sum\limits_{K<\mathcal{N}(m)\le K'}
\min\left\{||\Re(jmc)||^{-1},\sqrt{x/K}\right\}^{1/2}\cdot  \min\left\{||\Im(jmc)||^{-1},\sqrt{x/K}\right\}^{1/2}\\
\ll & x^{\varepsilon} \cdot \left(x/K\right)^{1/2}\cdot \sum\limits_{0<\mathcal{N}(l)\le (H_1^2+H_2^2)K'}
\min\left\{||\Re(lc)||^{-1},\sqrt{x/K}\right\}^{1/2}\cdot \min\left\{||\Im(lc)||^{-1},\sqrt{x/K}\right\}^{1/2}\\
\ll & x^{\varepsilon}\cdot \left(x/K\right)^{1/2} \cdot 
\left(1+\frac{(H_1^2+H_2^2)K}{|q|^2}\right)\left(\left(x/K\right)^{1/2}+|q|^2\right)\\
\ll & x^{\varepsilon} \cdot
\left(xK^{-1}+(H_1^2+H_2^2)x|q|^{-2}+(H_1^2+H_2^2)x^{1/2}K^{1/2}+|q|^2x^{1/2}K^{-1/2}\right).
\end{split}
\end{equation}
If $(H_1^2+H_2^2)K'\le |q|^2/8$, then using \eqref{secondcase} instead of \eqref{firstcase}, we obtain
\begin{equation}
\begin{split}
E_3(H_1,H_2,K,K') \ll (x|q|)^{\varepsilon}\left(|q|x^{3/4}K^{-3/4}+|q|^2x^{1/2}K^{-1/2}\right).
\end{split}
\end{equation}
We deduce that for all $K\ge 1/2$, 
\begin{equation}
\begin{split}
& E_3(H_1,H_2,K,K')\ll (x|q|)^{\varepsilon} \times\\ &  \left((H_1^2+H_2^2)x|q|^{-2}+(H_1^2+H_2^2)x^{1/2}K^{1/2}+|q|x^{3/4}K^{-3/4}+
|q|^2x^{1/2}K^{-1/2}\right)
\end{split}
\end{equation}
which implies
\begin{equation} \label{E3end}
\begin{split}
& E_3(H_1,H_2)\ll (x|q|)^{\varepsilon}\times\\
& \left((H_1^2+H_2^2)x|q|^{-2}+(H_1^2+H_2^2)x^{1/2}M^{1/2}+|q|x^{3/4}+|q|^2x^{1/2}\right).
\end{split}
\end{equation}

Now, from \eqref{E3rel} and \eqref{E3end}, we obtain
\begin{equation} \label{E3ende}
E_3\ll (Jx|q|)^{\varepsilon} \cdot \left(\delta^2 J^2x|q|^{-2}+\delta^2 J^2x^{1/2}M^{1/2}+\delta^2|q|x^{3/4}+\delta^2|q|^2x^{1/2}\right),
\end{equation}
where we use the inequality
$$
\min\{\delta,H_1^{-1}\}\cdot \min\{\delta,H_2^{-1}\} \le \delta^2,
$$
and from \eqref{F3rel} and \eqref{F3est}, we obtain
\begin{equation} \label{F3ende}
F_3 \ll (Jx)^{\varepsilon}\left(x^{(1+\alpha+\beta)/2}+
\delta Jx|q|^{-1}+\delta J x^{1-\alpha/4}+\delta |q| x^{1/2+(\alpha+\beta)/4}\right),
\end{equation}
where we use the inequalities
$$
\min\{\delta,H_1^{-1}\}\cdot \min\{\delta,H_2^{-1}\} \le (H_1H_2)^{-1}
$$
and
$$
\min\{\delta,H_1^{-1}\}\cdot \min\{\delta,H_2^{-1}\} \le \delta(H_1H_2)^{-1/2}
$$
(the first for the diagonal, the second for the non-diagonal contribution).

Further, from \eqref{E1E2}, \eqref{E1H}, \eqref{E1E3} and \eqref{E3end}, we infer
\begin{equation} \label{E1E2ende}
E_1, E_2\ll (Jx|q|)^{\varepsilon} \cdot \left(\delta J^2x|q|^{-2}+\delta J^2x^{1/2}M^{1/2}+\delta|q|x^{3/4}+\delta |q|^2x^{1/2}\right), 
\end{equation}
where we use the inequality
$$
\min\{\delta,H^{-1}\}\le \delta,
$$
and from \eqref{F1F2}, \eqref{F1H}, \eqref{F1F3} and \eqref{F3ende}, we infer
\begin{equation} \label{F1F2ende}
\begin{split}
F_1, F_2\ll (Jx|q|)^{\varepsilon}\cdot
\left(x^{(1+\alpha+\beta)/2}+
\delta^{1/2} Jx|q|^{-1}+\delta^{1/2} J x^{1-\alpha/4}+\delta^{1/2} |q| x^{1/2+(\alpha+\beta)/4}\right),
\end{split}
\end{equation}
where we use the inequalities
$$
\min\{\delta,H^{-1}\}\le H^{-1}\quad \mbox{and} \quad \min\{\delta,H^{-1}\}\le \delta^{1/2}H^{-1/2}.
$$

Combing \eqref{typeI}, \eqref{E3ende} and \eqref{E1E2ende}, we obtain
\begin{equation} \label{typeIest}
\begin{split}
& \sum\limits_{\substack{\mathcal{N}(m)\le M\\ mn\in A}} a_m = 4\delta^2 \sum\limits_{\substack{\mathcal{N}(m)\le M\\ mn\in B}} a_m+\\ &
O\left((Jx|q|)^{\varepsilon}\cdot\left(\delta x J^{-1}+\delta^2 J^2x|q|^{-2}+\delta^2 J^2x^{1/2}M^{1/2}+
\delta^2|q|x^{3/4}+\delta^2|q|^2x^{1/2}\right)\right), 
\end{split}
\end{equation}
and combining \eqref{typeII}, \eqref{F3ende} and \eqref{F1F2ende}, we obtain
\begin{equation} \label{typeIIest}
\begin{split}
& \sum\limits_{\substack{x_2^{\alpha}<\mathcal{N}(m)\le x_2^{\alpha+\beta}\\ mn\in A}} a_mb_n = 
4\delta^2 \sum\limits_{\substack{\mathcal{N}(m)\le M\\ mn\in B}} a_mb_n+ \\ &
O\left((Jx)^{\varepsilon}\cdot \left(\delta xJ^{-1}+x^{(1+\alpha+\beta)/2}+
\delta Jx|q|^{-1}+\delta J x^{1-\alpha/4}+\delta |q| x^{1/2+(\alpha+\beta)/4}\right)\right). 
\end{split}
\end{equation}

Now we choose $J:=[\delta^{-1}x^{3\varepsilon}]$, $x:=|q|^{12}$ (and hence $|q|:=x^{1/12}$), $M= x^{2/3}$, 
$\alpha:=1/3$ and $\beta:=1/2$ so that
\begin{equation} \label{typeIesti}
\sum\limits_{\substack{\mathcal{N}(m)\le M\\ mn\in A}} a_m = 4\delta^2 \sum\limits_{\substack{\mathcal{N}(m)\le M\\ mn\in B}} a_m+
O\left(\delta^2x^{1-\varepsilon}+x^{5/6+8\varepsilon}\right) 
\end{equation}
and 
\begin{equation} \label{typeIIesti}
\begin{split}
\sum\limits_{\substack{x_2^{\alpha}<\mathcal{N}(m)\le x_2^{\alpha+\beta}\\ mn\in A}} a_mb_n = 
4\delta^2 \sum\limits_{\substack{\mathcal{N}(m)\le M\\ mn\in B}} a_mb_n+ O\left(\delta^2x^{1-\varepsilon}+x^{11/12+8\varepsilon}\right).
\end{split}
\end{equation}

\subsection{Conclusion} \label{con}
Having proved \eqref{typeIesti} and \eqref{typeIIesti}, we deduce that
\eqref{t1} and \eqref{t2} hold with $Y=\delta^2x_2^{1-\varepsilon}$ if $\delta\ge x_2^{-1/24+5\varepsilon}$. 
Now using Theorem \ref{Harsie}, \eqref{primenumber}, \eqref{Sxd} and \eqref{Sx}, it follows that
$$
\sum\limits_{\substack{x_1<\mathcal{N}(p)\le x_2\\ ||pc ||< \delta\\ \omega_1<\arg p \le \omega_2}} 1 =
4\delta^2 \sum\limits_{\substack{x_1<\mathcal{N}(p)\le x_2}} 1 + O\left(x_2^{1/2}\right),
$$
provided that $x_2=|q|^{12}$, where $a/q$ is a Hurwitz continued fraction approximant of $c$ and
$\delta\ge x_2^{-1/24+\varepsilon}$ for any fixed $\varepsilon>0$. So by taking $N_k=|q_k|^{12}$, where $q_k$ is the $k$-th Hurwitz continued fraction denominator for $c$, 
we have the following result.

\begin{theorem} \label{approx}
Let $c$ be a complex number such that $c\not\in \mathbb{Q}(i)$, $\varepsilon>0$ be an arbitrary 
constant and $-\pi\le \omega_1< \omega_2\le \pi$. 
Then there exists an infinite increasing sequence of natural numbers $(N_k)_{k\in \mathbb{N}}$ such that 
\begin{equation*}
\sum\limits_{\substack{x<\mathcal{N}(p)\le N_k\\ ||pc ||\le\delta_k\\ \omega_1<\arg p\le \omega_2}} 1 =
4\delta_k^2 \sum\limits_{\substack{x<\mathcal{N}(p)\le N_k\\ \omega_1<\arg p\le \omega_2}} 1 + O\left(N_k^{1/2}\right)
\quad \mbox{as } k\rightarrow \infty
\end{equation*}
if $x<N_k$ and $N_k^{-1/24+\varepsilon}\le \delta_k\le 1/2$. 
\end{theorem}

This is implies Theorems \ref{coro} and \ref{signi1}. 

\subsection{Notes}
(I) The bound \eqref{lin} for linear exponential sums over $\mathbb{Z}[i]$ was obtained in a very simple way by reduction to one-dimensional 
linear exponential sums. Certainly, refinements are possible under certain conditions, and this may be useful for other applications. However,
it seems that improvements of \eqref{lin} and the subsequent bounds for averages of linear exponential sums don't help in this context
because the terms that dominate here cannot be removed, in particular, the term $\frac{z}{|q|^2}\cdot y^{1/2}$ in \eqref{firstcase}. 
So improvements of \eqref{lin}
will most likely not lead to progress with regard to the problem considered here.\\ 

(II) It should be possible to improve the exponent 1/12 in Theorems \ref{complex} and \ref{coro} using lower bound sieves.
To improve this exponent in the asymptotic relation in Theorem \ref{signi1} as well, different
techniques (like bounds for Koosterman-type sums) will be required. This may be an interesting
line of future research.\\

(III) Another interesting line could be to investigate Diophantine approximation problems of
this type for general number fields.

\section{Proof of Theorem \ref{Theo}(i)}
In this section, we prove Theorem \ref{Theo}(i).
Following the treatment in section 3, an admissible choice for the $M_k$'s in Theorem \ref{signi1} and Corollary \ref{signi} are the 
sixth powers of absolute values of the Hurwitz continued fraction approximants of $c$. Here we note that $M_k:=N_k^{1/2}$. 
Throughout the remainder of this paper, 
we assume that the $M_k$'s are of this form, set 
$$
\mathcal{S}:=\{M_1,M_2,...\}
$$
and suppose that $N\in \mathcal{S}$. Further, we write
$$
B_{\Delta}\left(a\right):= \left\{x\in \mathbb{C}\ :\ |x-a|\le \Delta\right\}
$$
and keep the notation $\mathbb{G}$ for the set of Gaussian primes. 
Let 
$$
\mathcal{A}_p=\bigcup_{\substack{r\in \mathbb{G}\\ q\in \mathbb{Z}[i]}} B_{|\eta/p|}\left(
\frac{r}{p}\right) \cap B_{|\eta/(cp)|}\left(\frac{1}{c}\cdot \frac{q}{p}\right) \cap D(a,b,\gamma_1,\gamma_2),
$$
where $\eta:=|p|^{\varepsilon-1/12}$. Then
\begin{equation} \label{integral}
\int\limits_{\gamma_1}^{\gamma_2}\int\limits_{a}^{b} F_N\left(re^{i\theta}\right)\ dr\ d\theta =
\sum\limits_{\substack{p \in \mathbb{G}\\ |p|\le N}} \mu(\mathcal{A}_p).
\end{equation}

Set
\begin{equation} \label{Mdefi}
M:=(a+b)/(2a)
\end{equation}
and 
$$
L:=4\left[\frac{2\pi}{\gamma_2-\gamma_1}\right]. 
$$
Our strategy is to split the summation over $p$ on the right-hand side of \eqref{integral} into summations over sets of the form 
$D(P,MP,\omega,\omega+2\pi/L)$ with $MP\le N$ and derive lower bounds. Clearly,
\begin{equation} \label{AB}
\sum\limits_{p \in \mathbb{G} \cap D(P,MP,\omega,\omega+2\pi/L) } \mu(\mathcal{A}_p)\ge
\sum\limits_{p \in \mathbb{G} \cap D(P,MP,\omega,\omega+2\pi/L) } \mu(\mathcal{B}_p)
\end{equation}
with
$$
\mathcal{B}_p=\bigcup_{\substack{r\in \mathbb{G}\\ q\in \mathbb{Z}[i]}} B_{|\eta'/p|}\left(
\frac{r}{p}\right) \cap B_{|\eta'/(cp)|}\left(\frac{1}{c}\cdot \frac{q}{p}\right) \cap D(a,b,\gamma_1,\gamma_2),
$$
where 
\begin{equation} \label{etadef}
\eta':=(MP)^{\varepsilon-1/12}.
\end{equation}

We note that if $|p|\le MP$ and
$$
\frac{1}{c}\cdot \frac{q}{p} \in B_{|\eta'/p|}\left(\frac{r}{p}\right),
$$
which latter is equivalent to
$$
q \in B_{|\eta'c|}(rc),
$$
then 
$$
\mu\left(B_{|\eta'/p|}\left(
\frac{r}{p}\right) \cap B_{|\eta'/(cp)|}\left(\frac{1}{c}\cdot \frac{q}{p}\right)\right)\ge \nu,
$$
where
\begin{equation} \label{nudef}
\nu:=\left(\frac{\pi}{3}-\frac{\sqrt{3}}{2}\right)\cdot \left|\frac{\eta'}{MP}\right|^2=\left(\frac{\pi}{3}-\frac{\sqrt{3}}{2}\right)\cdot (MP)^{2\varepsilon-13/6}.
\end{equation}
Here we use our condition that $0<|c|\le 1$.
Also, for all $p\in D(P,MP,\omega,\omega+2\pi/L)$,
$$
r\in D(MPa,Pb,\gamma_1+\omega+2\pi/L,\gamma_2+\omega) \Longrightarrow \frac{r}{p}\in D(a,b,\gamma_1,\gamma_2).
$$
We thus have 
\begin{equation} \label{key}
\sum\limits_{p\in D(P,MP,\omega,\omega+2\pi/L)} \mu(\mathcal{B}_p) \ge \nu N(P,\omega),
\end{equation}
where $N(P,\omega)$ counts the number of $(p,q,r)\in \mathbb{G} \times \mathbb{Z}[i]\times \mathbb{G}$ satisfying 
\begin{equation} \label{qprconds}
 p\in D(P,MP,\omega,\omega+2\pi/L), \quad q\in B_{\delta}(rc), \quad r\in D(MPa,Pb,\gamma_1+\omega+2\pi/L,\gamma_2+\omega),  
\end{equation}
where
\begin{equation} \label{deltadef}
\delta:=|\eta' c|=\frac{|c|}{(MP)^{1/12-\varepsilon}}.
\end{equation}
We note that
\begin{equation} \label{note1}
\frac{1}{4}\cdot \left(b^2-a^2\right)\cdot P^2 \le (Pb)^2-(MPa)^2=\frac{3}{4}\cdot \left(b^2-a^2\right) \cdot P^2
\end{equation}
and
\begin{equation} \label{note2}
\frac{\gamma_2-\gamma_1}{2}\le (\gamma_2+\omega)-(\gamma_1+\omega+2\pi/L)\le \gamma_2-\gamma_1.
\end{equation}

Using Theorem \ref{PNT}, the number $\pi(P,MP,\omega,\omega+2\pi/L)$ of Gaussian primes \\
$p\in D(P,MP,\omega,\omega+2\pi/L)$ is bounded from below by 
\begin{equation} \label{pcount}
\pi(P,MP,\omega,\omega+2\pi/L) \ge \frac{4}{L}\cdot \frac{(M^2-1)P^2+o(L(MP)^2)}{2\log N}.
\end{equation}
The number of $(q,r)\in \mathbb{Z}[i]\times \mathbb{G}$ satisfying
$$
q\in B_{\delta}(rc), \quad r\in D(MPa,Pb,\gamma_1+\omega+2\pi/L,\gamma_2+\omega)
$$ 
equals $\pi_c(MPa,Pb,\gamma_1+\omega+2\pi/L,\gamma_2+\omega;\delta)$ and is, by Corollary \ref{signi}, bounded from below by
\begin{equation} \label{qrcount}
\begin{split}
& \pi_c(MPa,Pb,\gamma_1+\omega+2\pi/L,\gamma_2+\omega;\delta)\\ 
\ge & 
\frac{\delta^2\left((\gamma_2+\omega)-(\gamma_1+\omega+2\pi/L)\right)((MPa)^2-(Pb)^2)+o(\delta^2 N^2)}{\log N}.
\end{split}
\end{equation}

Combining \eqref{AB}, \eqref{nudef}, \eqref{key}, \eqref{deltadef}, \eqref{note1}, \eqref{note2}, \eqref{pcount} and \eqref{qrcount}, we obtain
\begin{equation} \label{near}
\begin{split}
& \sum\limits_{p \in \mathbb{G} \cap D(P,MP,\omega,\omega+2\pi/L) } \mu(\mathcal{A}_p)\\ \ge & 
C(MP)^{4\varepsilon-7/3}\cdot \frac{1}{L}\cdot \frac{(M^2-1)P^2+o(L(MP)^2)}{\log N}\cdot 
\frac{(\gamma_2-\gamma_1)(b^2-a^2)P^2+o(N^2)}{\log N}
\end{split}
\end{equation}
for some constant $C=C(c)>0$ if $N\in \mathcal{S}$ and $N\rightarrow \infty$.
By splitting the interval $(1,N]$ into intervals of the form $(P,MP]=(N/M^k,N/M^{k-1}]$ and summing up, it follows from \eqref{near} that
\begin{equation*}
\begin{split}
& \sum\limits_{p \in \mathbb{G} \cap D(1,N,\omega,\omega+2\pi/L)} \mu(\mathcal{A}_p)\\ \ge & 
\frac{C(M^2-1)(\gamma_2-\gamma_1)(b^2-a^2)N^{5/3+4\varepsilon}}{M^{7/3-4\varepsilon}L\log^2 N} \cdot \sum\limits_{k=1}^{\infty} M^{-(5/3+4\varepsilon)k} \cdot (1+o(1))\\
\ge & \frac{C(\gamma_2-\gamma_1)(b^2-a^2)N^{5/3+4\varepsilon}}{M^{2/3}L\log^2 N} 
\cdot (1+o(1))\\
\ge & C\cdot \frac{A}{B} \cdot \frac{(\gamma_2-\gamma_1)(b^2-a^2)N^{5/3+4\varepsilon}}{L\log^2 N}
\end{split}
\end{equation*}
if $N\in \mathcal{S}$ is large enough, where for the last line, we have used \eqref{Mdefi} and $A\le a<b\le B$.   
Splitting the interval $(0,2\pi]$ into intervals of the form $(\omega,\omega+2\pi/L]=(2\pi (k-1)/L,2\pi k/L]$ with $k=1,...,L$, it further follows that
$$
\sum\limits_{\substack{p \in \mathbb{G}\\ |p|\le N}} \mu(\mathcal{A}_p) = 
\sum\limits_{k=1}^L \sum\limits_{p \in \mathbb{G} \cap D(1,N,2\pi (k-1)/L,2\pi k/L)} \mu(\mathcal{A}_p) \ge
C\cdot \frac{A}{B}  \cdot \frac{(\gamma_2-\gamma_1)(b^2-a^2)N^{5/3+4\varepsilon}}{\log^2 N}
$$
if $N$ is large enough. Combining this with \eqref{integral} completes the proof of Theorem \ref{Theo}(i).  

\section{Proof of Theorem \ref{Theo}(ii)} 
In this section, we prove Theorem \ref{Theo}(ii), which is the last task to establish our main result, Theorem \ref{complex}.
\subsection{Sieve theoretical approach}
We extend the treatment in \cite[section 5]{BG1} (see also \cite[section 4]{BG2}), which has its origin in \cite{HJ}, to the situation in $\mathbb{Z}[i]$.  
We point out that there is a mistake in \cite[section 5]{BG1}: The set $\mathcal{A}$ should consist of products of the form
$n[n\alpha]$, not of the form $n[n\alpha][nc\alpha]$, and 
we bound the number of $n$'s such that $n[n\alpha]$ is the product of two primes, not three primes. This mistake, however, doesn't 
affect the method and the final result. As in section 3, we define 
$$
||z||:=\max\left\{||\Re(z)||,||\Im(z)||\right\},
$$
where $\Re(z)$ is the real part and $\Im(z)$ is the imaginary part of $z\in \mathbb{C}$, and for $x\in \mathbb{R}$, $||x||$ denotes the distance of $x$ to the nearest integer. We further define
$$
f(z):=\tilde{f}(\Re(z))+\tilde{f}(\Im(z))i
$$
if $z\in \mathbb{C}$ and $||z||<1/2$, where $\tilde{f}(x)$ is the integer nearest to $x\not\in \mathbb{Z}+1/2$. 

First, we split the interval $(0,N]$ into dyadic intervals $(P/2,P]$ with $P:=N/2^k$, $k=0,1,2,...$. Then we write 
$$
\mathcal{A}_P(\alpha)=\left\{n\cdot f(n\alpha) \ :\  n\in \mathbb{Z}[i], P/2< |n|\le P, \ \max\{||n\alpha||,||nc\alpha||\}\le
\mu\right\},
$$
where
\begin{equation} \label{muedef}
\mu:=\left(\frac{P}{2}\right)^{\varepsilon-1/12}
\end{equation}
if $P^{\varepsilon-1/12}<1/2$, i.e.
\begin{equation} \label{Pcondit}
P>2^{1+1/(1/12-\varepsilon)}.
\end{equation}
It follows that
\begin{equation} \label{upperFN}
F_N(\alpha)\le \sum\limits_{\substack{0\le k\le 1+\log_2\left(N/2^{1/(1/12-\varepsilon)}\right)}}
\sharp \left(\mathbb{G}_2\cap \mathcal{A}_{N/2^k}(\alpha)\right)+O(1),
\end{equation}
where $\mathbb{G}_2$ is the set of products of two Gaussian primes. We bound 
$\sharp\left(\mathbb{G}_2\cap \mathcal{A}_{N/2^k}(\alpha)\right)$ from above using a simple two-dimensional upper bound 
sieve in the setting of Gaussian integers, which is obtained by a standard application of the Selberg sieve in the setting of Gaussian
integers. \\

\begin{lemma} \label{uppersieve} Let $\mathcal{N}$ be a subset of the Gaussian integers and $f_1,f_2:\mathcal{N} \rightarrow \mathbb{Z}[i]$ two functions. 
For $P\ge 2$ and $d_1,d_2\in \mathbb{Z}[i]\setminus\{0\}$ let $T_P(d_1,d_2)$ be the number of $n\in \mathcal{N}$ such that
\begin{equation} \label{satisfy}
f_1(n)\equiv 0 \bmod{d_1}, \quad f_2(n)\equiv 0 \bmod{d_2}, \quad P/2< |n|\le P
\end{equation}
and $G_P(d_1,d_2)$ the number of $n\in \mathcal{N}$ satisfying \eqref{satisfy} such that $f_1(n)$ and $f_2(n)$ are both Gaussian primes. Then
for any $X>0$ and $\varepsilon>0$,
$$
G_P(d_1,d_2)\le \frac{C(\varepsilon)XP^2}{(\log P)^2}+ 
O\left(\sum\limits_{\substack{d_1,d_2\in \mathbb{Z}[i]\setminus \{0\}\\ 1\le |d_1|,|d_2|\le P^{\varepsilon}}} 
|d_1d_2|^{\varepsilon}\left|T_P(d_1,d_2) - \frac{XP^2}{|d_1|^2|d_2|^2}\right|\right)
$$
as $P\rightarrow \infty$, where $C(\varepsilon)$ is a constant depending only on $\varepsilon$. 
\end{lemma}

Here we consider the case when
\begin{equation} \label{max}
\mathcal{N}:=\{n\in \mathbb{Z}[i] : \max(||n\alpha||,||nc\alpha||)\le \mu\},
\end{equation}
$f_1(n)=n$ and $f_2(n)=f(n\alpha)$. We write $T_P(\alpha;d_1,d_2):=T_P(d_1,d_2)$. Clearly, $T_P(\alpha;d_1,d_2)$ equals the number of
$n\in \mathbb{Z}[i]$ with $P/2<|n|\le P$ such that
\begin{equation} \label{threeconds}
\frac{P}{2|d_1|}<|n|\le \frac{P}{|d_1|}, \quad \left|\left|\frac{nd_1\alpha}{d_2}\right|\right|\le \frac{\mu}{|d_2|}, \quad 
||nd_1c\alpha||\le \mu.
\end{equation}
Heuristically, $T_P(\alpha;d_1,d_2)$ should behave like $12\pi P^2\mu^4/(|d_1|^2|d_2|^2)$. Therefore, we write 
\begin{equation} \label{write}
T_P(\alpha;d_1,d_2)=\frac{12\pi P^2\mu^4}{|d_1|^2|d_2|^2}+ E_P(\alpha;d_1,d_2).
\end{equation}
Then, applying Lemma \ref{uppersieve} gives
\begin{equation}
\sharp\left(\mathbb{G}_2\cap \mathcal{A}_P\right)\le \frac{C(\varepsilon)P^2\mu^4}{\log^2 P} + O\left(\tilde{J}_P(\alpha)\right),
\end{equation}
where 
\begin{equation*}
\begin{split}
\tilde{J}_P(\alpha):= \sum\limits_{1\le |d_1d_2|\le P^{\varepsilon}} |d_1d_2|^{\varepsilon}|E_P(\alpha;d_1,d_2)|.
\end{split}
\end{equation*}
Hence, by \eqref{upperFN}, to establish the claim in Theorem \ref{Theo}(ii), it suffices to show that
\begin{equation} \label{average}
\begin{split}
& \sum\limits_{\substack{0\le k\le 1+\log_2(N/2^{1/(1/12-\varepsilon)})}}\sum\limits_{1\le |d_1d_2|\le N^{\varepsilon}} |d_1d_2|^{\varepsilon}\int\limits_{-\pi}^{\pi} \int\limits_A^B 
\left|E_{N/2^k}(Re^{i\theta};d_1,d_2)\right|\ dR \ d\theta\\  = & 
o\left(\frac{N^2\mu^4}{\log^2 N}\right) 
\end{split}
\end{equation}
as $N\rightarrow \infty$ and $N\in \mathcal{S}$.

\subsection{Fourier analysis}
Throughout the sequel, we assume that \eqref{Pcondit} is satisfied. 
We use Fourier analysis to express $E_P(\alpha;d_1,d_2)$ in terms of trigonometrical polynomials. We have

\begin{equation*}
\begin{split}
T_P(\alpha;d_1,d_2)= \sum\limits_{\substack{n\in \mathbb{Z}[i]\\ P/|2d_1|<|n|\le P/|d_1|}} &
\left(\left[\Re\left(\frac{nd_1\alpha}{d_2}\right)+\frac{\mu}{|d_2|}\right]-\left[\Re\left(\frac{nd_1\alpha}{d_2}\right)-\frac{\mu}{|d_2|}
\right]
\right)\times\\ &
\left(\left[\Im\left(\frac{nd_1\alpha}{d_2}\right)+\frac{\mu}{|d_2|}\right]-\left[\Im\left(\frac{nd_1\alpha}{d_2}\right)-\frac{\mu}{|d_2|}\right]
\right)\times\\ & 
\left(\left[\Re\left(nd_1c\alpha\right)+\mu\right]-\left[\Re\left(nd_1c\alpha\right)-\mu\right]\right)\times\\ &
\left(\left[\Im\left(nd_1c\alpha\right)+\mu\right]-\left[\Im\left(nd_1c\alpha\right)-\mu\right]\right),
\end{split}
\end{equation*}
where $[x]$ is the integral part of $x\in \mathbb{R}$. 
Writing $\psi(x)=x-[x]-1/2$, it follows that
\begin{equation*}
\begin{split}
& T_P(\alpha;d_1,d_2)= \\ & \sum\limits_{\substack{n\in \mathbb{Z}[i]\\ P/|2d_1|<|n|\le P/|d_1|}} 
\left(\psi\left(\Re\left(\frac{nd_1\alpha}{d_2}\right)-\frac{\mu}{|d_2|}\right)-
\psi\left(\Re\left(\frac{nd_1\alpha}{d_2}\right)+\frac{\mu}{|d_2|}\right)
+\frac{2\mu}{|d_2|}\right)\times\\ &
\left(\psi\left(\Im\left(\frac{nd_1\alpha}{d_2}\right)-\frac{\mu}{|d_2|}\right)-
\psi\left(\Im\left(\frac{nd_1\alpha}{d_2}\right)+\frac{\mu}{|d_2|}\right)+\frac{2\mu}{|d_2|}\right)\times\\ & 
\left(\psi\left(\Re\left(nd_1c\alpha\right)-\mu\right)-\psi\left(\Re\left(nd_1c\alpha\right)+\mu\right)+2\mu\right)\times\\ &
\left(\psi\left(\Im\left(nd_1c\alpha\right)-\mu\right)-\psi\left(\Im\left(nd_1c\alpha\right)+\mu\right)+2\mu\right).
\end{split}
\end{equation*}

Let 
 \begin{equation} \label{J1J2def}
 J_1:= \left[\frac{N^{\varepsilon}|d_2|}{\mu}\right]  \quad \mbox{and} \quad  J_2:=\left[\frac{N^{\varepsilon}}{\mu}\right].
 \end{equation}
 Then from Lemma \ref{Vaaler}, using
$$
\frac{1}{j}\cdot (e(jx)-e(-jx))\ll x
$$
for any $j\in \mathbb{N}$ and $x\in \mathbb{R}$, we deduce that
\begin{equation} \label{deduce}
\begin{split}
T_P(\alpha;d_1,d_2)= & \frac{16\mu^4}{|d_2|^2}\cdot \sum\limits_{\substack{n\in \mathbb{Z}[i]\\ P/|2d_1|<|n|\le P/|d_1|}} 1 + O\left(1+\frac{P^2}{J_1^2J_2^2|d_1|^2}+F_P(\alpha;d_1,d_2)\right)\\
= & \frac{12\pi P^2\mu^4}{|d_1|^2 |d_2|^2} + O\left(\frac{P\mu^2}{|d_1| |d_2|}+\frac{P^2}{J_1^2J_2^2|d_1|^2}+F_P(\alpha;d_1,d_2)\right), 
\end{split}
\end{equation}
where 
\begin{equation*}
\begin{split}
F_P(\alpha;d_1,d_2)= & \frac{\mu^4}{|d_2|^2}\cdot \sum\limits_{\substack{(m_1,m_2,m_3,m_4)\in \mathbb{Z}^4\setminus\{(0,0,0,0)\}
\\ |m_1|\le J_1,\ |m_2|\le J_1\\ |m_3|\le J_2,\ |m_4|\le J_2}} 
\\ & \Big| \sum\limits_{\substack{n\in \mathbb{Z}[i]\\ P/|2d_1|<|n|\le P/|d_1|}} e\Big(m_1\cdot \Re\Big(\frac{nd_1\alpha}{d_2}\Big)+ m_2\cdot 
\Im\Big(\frac{nd_1\alpha}{d_2}\Big)\\ & +m_3\cdot \Re(nd_1c\alpha)+ m_4\cdot 
\Im(nd_1c\alpha)\Big) \Big|\\
= & \frac{\mu^4}{|d_2|^2}\cdot 
\sum\limits_{\substack{(m_1,m_2,m_3,m_4)\in \mathbb{Z}^4\setminus\{(0,0,0,0)\}\\ |m_1|\le J_1,\ |m_2|\le J_1\\ |m_3|\le J_2,\ |m_4|\le J_2}} 
\\ & \Big| \sum\limits_{\substack{n\in \mathbb{Z}[i]\\ P/|2d_1|<|n|\le P/|d_1|}} e\Big(\Im\Big(nd_1\alpha \cdot 
\Big(\frac{m_1+im_2}{d_2}+(m_3+im_4)c\Big)\Big)\Big) \Big|.
\end{split}
\end{equation*}
For the last line of \eqref{deduce}, we have used the elementary bound for the error term in the Gauss circle problem, namely
$$
\sum\limits_{\substack{n\in \mathbb{Z}[i]\\ |n|\le x}} 1 = \pi x^2+O(x).
$$
In the following, we will prove that
\begin{equation}  \label{end}
\int\limits_{-\pi}^{\pi} \int\limits_A^B 
\left|F_{P}(Re^{i\theta};d_1,d_2)\right|\ dR \ d\theta \ll \frac{N^{2-4\varepsilon}\mu^4}{|d_1|^2|d_2|^2}.
\end{equation} 
In view of \eqref{muedef}, \eqref{write}, \eqref{J1J2def} and \eqref{deduce},  this suffices to prove \eqref{average} and therefore establishes the claim of 
Theorem \ref{Theo}(ii). 

We first bound $F_P(\alpha;d_1,d_2)$ for individual $\alpha$. Using \eqref{lin} with $f_1=-\pi$ and $f_2=\pi$, we have
\begin{equation} \label{lino}
\sum\limits_{\substack{n\in \mathbb{Z}[i]\\ \tilde{x}<|n|\le x}} e\left(\Im(n\kappa)\right) \ll x \cdot 
\min\left\{||\Im(\kappa)||^{-1},x\right\}^{1/2} \cdot \min\left\{||\Re(\kappa)||^{-1},x\right\}^{1/2}.
\end{equation}
It follows that
\begin{equation} \label{follows}
\begin{split}
& F_P(\alpha;d_1,d_2)\\ \ll &  
\frac{P\mu^4}{|d_1|\cdot |d_2|^2}\cdot  
\sum\limits_{\substack{(m_1,m_2,m_3,m_4)\in \mathbb{Z}^4\setminus\{(0,0,0,0)\}\\ |m_1|\le J_1,\ |m_2|\le J_1\\ |m_3|\le J_2,\ |m_4|\le J_2}}
\\ & \min\left\{\left|\left|\Im\left(d_1\Big(\frac{m_1+im_2}{d_2}+(m_3+im_4)c\Big)\alpha\right)\right|\right|^{-1}, \frac{P}{|d_1|}\right\}^{1/2} \times\\
& \min\left\{\left|\left|\Re\left(d_1\Big(\frac{m_1+im_2}{d_2}+(m_3+im_4)c\Big)\alpha\right)\right|\right|^{-1}, \frac{P}{|d_1|}\right\}^{1/2}\\
\le & \frac{P\mu^4}{|d_1|\cdot |d_2|^2}\cdot
\sum\limits_{\substack{(n_1,n_2)\in \mathbb{Z}[i]^2\setminus\{(0,0)\}\\ |n_1|\le 2J_1\\ |n_2|\le 2J_2}}
\min\left\{\left|\left|\Im\left(d_1\Big(\frac{n_1}{d_2}+n_2c\Big)\alpha\right)\right|\right|^{-1}, \frac{P}{|d_1|}\right\}^{1/2} \times\\
& \min\left\{\left|\left|\Re\left(d_1\Big(\frac{n_1}{d_2}+n_2c\Big)\alpha\right)\right|\right|^{-1}, \frac{P}{|d_1|}\right\}^{1/2}.\\
\end{split}
\end{equation}

\subsection{Average estimation for $F_P(\alpha;d_1,d_2)$} 
To bound the double integral on the left-hand side of \eqref{end}, 
we now use the following lemma. 

\begin{lemma}\label{av} Let $z\in \mathbb{C}$ and $Y>0$. Then
\begin{equation*}
\begin{split}
& \int\limits_{-\pi}^{\pi} \int\limits_{A}^{B} \min\left\{\left|\left|\Im\left(zRe^{\theta i}\right)\right|\right|^{-1}, Y\right\}^{1/2} \cdot
\min\left\{\left|\left|\Re\left(zRe^{\theta i}\right)\right|\right|^{-1}, Y\right\}^{1/2} \ dR \ d\theta\\
\ll_{A,B}  & \max\left\{1,|z|^{-1}\right\}\log (2+Y).
\end{split}
\end{equation*}
\end{lemma}

\begin{proof} By change of variables, we have
\begin{equation*}
\begin{split}
& \int\limits_{-\pi}^{\pi} \int\limits_{A}^{B} \min\left\{\left|\left|\Im\left(zRe^{\theta i}\right)\right|\right|^{-1}, Y\right\}^{1/2} \cdot
\min\left\{\left|\left|\Re\left(zRe^{\theta i}\right)\right|\right|^{-1}, Y\right\}^{1/2} \ dR \ d\theta\\ = &  
\frac{1}{|z|^2}\cdot \int\limits_{-\pi}^{\pi} \int\limits_{|z|A}^{|z|B} \min\left\{\left|\left|\Im\left(r e^{\theta i}\right)\right|\right|^{-1}, Y\right\}^{1/2} \cdot
\min\left\{\left|\left|\Re\left(re^{\theta i}\right)\right|\right|^{-1}, Y\right\}^{1/2} \ dr \ d\theta.
\end{split}
\end{equation*}
Changing from polar to affine coordinates, and using Cauchy-Schwarz, we get 
\begin{equation*}
\begin{split}
&  \frac{1}{|z|^2} \cdot \int\limits_{-\pi}^{\pi} \int\limits_{|z|A}^{|z|B} \min\left\{\left|\left|\Im\left(re^{\theta i}\right)\right|\right|^{-1}, Y\right\}^{1/2} \cdot
\min\left\{\left|\left|\Re\left(re^{\theta i}\right)\right|\right|^{-1}, Y\right\}^{1/2} \ dr \ d\theta\\
\le & \frac{1}{|z|^2}\cdot \int\limits_{|z|A/2}^{|z|B} \int\limits_{|z|A/2}^{|z|B}\min\left\{\left|\left|\Im\left(x+yi\right)\right|\right|^{-1}, Y\right\}^{1/2} \cdot
\min\left\{\left|\left|\Re\left(x+yi\right)\right|\right|^{-1}, Y\right\}^{1/2} \ dy \ dx\\
= & \frac{1}{|z|^2}\cdot \int\limits_{|z|A/2}^{|z|B}  \int\limits_{|z|A/2}^{|z|B}\min\left\{\left|\left|\Im\left(x+yi\right)\right|\right|^{-1}, Y\right\}^{1/2} \cdot
\min\left\{\left|\left|\Re\left(x+yi\right)\right|\right|^{-1}, Y\right\}^{1/2} \ dy \ dx\\
= & \frac{1}{|z|^2}\cdot \left( \int\limits_{|z|A/2}^{|z|B}  \min\left\{\left|\left| x\right|\right|^{-1}, Y\right\}^{1/2} \ dx\right)^2\\
\ll & \frac{1}{|z|}\cdot  \left(B-\frac{A}{2}\right)\cdot \int\limits_{|z|A/2}^{|z|B}  \min\left\{\left|\left| x\right|\right|^{-1}, Y\right\} \ dx\\
= &  \left(B-\frac{A}{2}\right)\cdot \int\limits_{A/2}^{B}  \min\left\{\left|\left| |z|x\right|\right|^{-1}, Y\right\} \ dx.
\end{split}
\end{equation*}
From Lemma 5.1 in \cite{BG1}, it follows that
$$
\int\limits_{A/2}^{B}  \min\left\{\left|\left|\ |z|x\ \right|\right|^{-1}, Y\right\} \ dx 
\ll_{A,B} \max\left\{1, \left|z\right|^{-1}\right\}\log (2+Y).
$$
Putting everything together proves the claim.
\end{proof}

From \eqref{follows} and Lemma \ref{av}, we deduce that
\begin{equation} \label{EP}
\begin{split}
& \int\limits_{-\pi}^{\pi} \int\limits_A^B 
\left|F_{P}(Re^{i\theta};d_1,d_2)\right|\ dR \ d\theta\\ \ll_{A,B} & \frac{P\mu^4\log P}{|d_1|\cdot |d_2|^2}\cdot 
\sum\limits_{\substack{(n_1,n_2)\in \mathbb{Z}[i]^2\setminus\{(0,0)\}\\ |n_1|\le 2J_1\\ |n_2|\le 2J_2}}
  \max\left\{1, \left|d_1\left(\frac{n_1}{d_2}+n_2c\right)\right|^{-1}\right\}.
\end{split}
\end{equation}

\subsection{Final estimation}
Clearly, if $0<|d_1|,|d_2|\le P^{\varepsilon}$ and $J_1\ge |d_2|/\mu$, then
\begin{equation} \label{EP1}
\begin{split}
& \sum\limits_{\substack{(n_1,n_2)\in \mathbb{Z}[i]^2\setminus\{(0,0)\}\\ |n_1|\le 2J_1\\ |n_2|\le 2J_2}} \max\left\{1, \left|d_1\left(\frac{n_1}{d_2}+n_2c\right)\right|^{-1}\right\}\\
\ll & \frac{1}{|d_1|}\cdot \sum\limits_{\substack{n_2\in \mathbb{Z}[i]\setminus\{0\}\\ |n_2|\le 2J_2}} \left(\min\limits_{n_1\in \mathbb{Z}[i]} \left|\frac{n_1}{d_2}+n_2c\right|\right)^{-1} +
J_2^2\cdot \sum\limits_{\substack{n\in \mathbb{Z}[i]\setminus\{0\}\\ |n|\le 3J_1}} \max\left\{1, \left|\frac{d_2}{d_1n}\right|\right\}\\
\ll & \frac{1}{|d_1|}\cdot \sum\limits_{\substack{n_2\in \mathbb{Z}[i]\setminus\{0\}\\ |n_2|\le 2J_2}} \left(\min\limits_{n_1\in \mathbb{Z}[i]} \left|\frac{n_1}{d_2}+n_2c\right|\right)^{-1} +
J_1^2J_2^2.
\end{split}
\end{equation}
Since $N$ is the sixth power of absolute value of a denominator of the Hurwitz continued fraction approximation of $c$, we have 
$$
c=\frac{a}{q}+O\left(\frac{1}{|q|^2}\right)
$$
for some $a,q\in \mathbb{Z}[i]$ with $|q|^6=N$. Hence,
$$
\left|\frac{n_1}{d_2}+n_2c\right| \ge \frac{1}{|d_2q|}+O\left(\frac{J_2}{|q|^2}\right).
$$
Since
$$
\frac{1}{|d_2q|}\ge P^{-\varepsilon}N^{-1/6} \ge N^{-1/6-\varepsilon},
$$
it follows that 
$$
\left|\frac{n_1}{d_2}+n_2c\right| \gg N^{-1/6-\varepsilon}
$$
if $n_1\in \mathbb{Z}[i]$, $n_2\in  \mathbb{Z}[i]\setminus\{0\}$, $|n_2|\le 2J_2$ and $N$ is large enough, 
where we recall that $J_2=[N^{\varepsilon}/\mu]\le N^{1/12+\varepsilon}$.  Hence, from \eqref{EP1}, we deduce that
\begin{equation} \label{EP2}
\sum\limits_{\substack{(n_1,n_2)\in \mathbb{Z}[i]^2\setminus\{(0,0)\}\\ |n_1|\le 2J_1\\ |n_2|\le 2J_2}} \max\left\{1, \left|d_1\left(\frac{n_1}{d_2}+n_2c\right)\right|^{-1}\right\}
\ll N^{1/6+\varepsilon}J_2^2 + J_1^2J_2^2.
\end{equation}

Combining \eqref{EP} and \eqref{EP2} , we obtain
\begin{equation*} 
\int\limits_{-\pi}^{\pi} \int\limits_A^B 
\left|F_{P}(Re^{i\theta};d_1,d_2)\right|\ dR \ d\theta\\ \ll_{A,B}  
\frac{P\mu^4\log P}{|d_1|\cdot |d_2|^2}\cdot \left(N^{1/6+\varepsilon}J_2^2 + J_1^2J_2^2\right),
\end{equation*}
from which \eqref{end} follows using \eqref{muedef} and \eqref{J1J2def}. This completes the proof of Theorem \ref{Theo}(ii).

\end{document}